\documentclass{article}
\usepackage[utf8]{inputenc}
\usepackage[margin= 1.75 cm]{geometry}

\usepackage{amsthm}
\usepackage{enumitem}
\usepackage{amsmath}

\usepackage{amssymb}
\usepackage{setspace}
\usepackage{mathtools}
\usepackage{verbatim}
\usepackage{csquotes}
\usepackage{graphicx}
\usepackage[hidelinks]{hyperref}
\usepackage{thm-restate}
\usepackage{cleveref}
\usepackage{graphicx}
\usepackage{mathrsfs}
\usepackage{enumitem}
\usepackage{framed}
\usepackage{subcaption}
\usepackage{crossreftools}
\usepackage{floatrow}
\usepackage[T1]{fontenc}
\usepackage{xcolor}

\theoremstyle{plain}

\newtheorem*{thm*}{Theorem}
\newtheorem{theorem}{Theorem}[section]
\Crefname{theorem}{Theorem}{Theorems}
\newtheorem*{lem*}{Lemma}
\newtheorem{lemma}[theorem]{Lemma}
\Crefname{lemma}{Lemma}{Lemmas}

\newtheorem*{claim*}{Claim}

\crefname{claim}{Claim}{Claims}
\Crefname{claim}{Claim}{Claims}

\Crefname{prop}{Proposition}{Propositions}

\crefname{corollary}{Corollary}{Corollaries}

\newtheorem{conj}[theorem]{Conjecture}
\crefname{conj}{Conjecture}{Conjectures}

\newtheorem*{conj*}{Conjecture}

\Crefname{qn}{Question}{Questions}

\newtheorem{obs}[theorem]{Observation}
\Crefname{obs}{Observation}{Observations}

\Crefname{ex}{Example}{Examples}

\theoremstyle{definition}
\newtheorem*{definition}{Definition}

\Crefname{problem}{Problem}{Problems}

\newtheorem{defn}[theorem]{Definition}
\Crefname{defn}{Definition}{Definitions}

\newtheorem*{defn*}{Definition}

\theoremstyle{remark}

\newtheorem{problem}[lemma]{Problem}
\newtheorem*{problem*}{Problem}

\renewenvironment{proof}[1][]{\begin{trivlist}
\item[\hspace{\labelsep}{\bf\noindent Proof#1.\/}] }{\qed\end{trivlist}}

\newenvironment{cla_proof}[1][]{\begin{proof}[#1]}{\end{proof}}

\newcommand{\eps}{\varepsilon}

\expandafter\def\expandafter\normalsize\expandafter{%
    \normalsize
    \setlength\abovedisplayskip{8pt}
    \setlength\belowdisplayskip{8pt}
    \setlength\abovedisplayshortskip{4pt}
    \setlength\belowdisplayshortskip{4pt}
}

\usepackage[square,sort,comma,numbers]{natbib}
 \setlist[itemize]{leftmargin=*}

\makeatletter
\newcommand{\optionaldesc}[2]{%
  \phantomsection
  #1\protected@edef\@currentlabel{#1}\label{#2}%
}
\makeatother

\newcounter{propcounter}

\title{Towards Graham's rearrangement conjecture via rainbow paths}
\author{Matija Buci\'c \thanks{Institute of Mathematics, University of Vienna, Vienna, Austria, and Department of Mathematics, Princeton University, Princeton, NJ, USA. Email: \href{mailto:matija.bucic@univie.ac.at} {\nolinkurl{matija.bucic@univie.ac.at}}. Research supported in part by NSF grants DMS--2349013 and DMS-1928930.} \and Bryce Frederickson\thanks{Department of Mathematics, Emory University, 
Atlanta, GA, USA. Email: {\tt bfrede4@emory.edu}. Research is supported by the NSF Graduate Research Fellowship Program under Grant No. 1937971.} \and Alp M\"uyesser \thanks{New College, University of Oxford, UK. Email: {\tt alp.muyesser@new.ox.ac.uk}. } \and  Alexey Pokrovskiy\thanks{Department of Mathematics, University College London, UK. Email: {\tt dralexeypokrovskiy@gmail.com}.} \and Liana Yepremyan \thanks{Department of Mathematics, Emory University, 
Atlanta, GA, USA. Email: {\tt lyeprem@emory.edu}. Research is supported by the National Science Foundation grant 2247013: Forbidden and Colored Subgraphs. }}

\date{}

\begin{document}
\maketitle
\begin{abstract}

We study an old question in combinatorial group theory which can be traced back to a conjecture of Graham from 1971. Given a group $\Gamma$, and some subset $S\subseteq \Gamma$, is it possible to permute $S$ as $s_1,s_2,\ldots, s_d$ so that the partial products $\prod_{1\leq i\leq t} s_i$, $t\in [d]$ are all distinct? Most of the progress towards this problem has been in the case when $\Gamma$ is a cyclic group. We show that for any group $\Gamma$ and any $S\subseteq \Gamma$, there is a permutation of $S$ where all but a vanishing proportion of the partial products are distinct, thereby establishing the first asymptotic version of Graham's conjecture under no restrictions on $\Gamma$ or $S$.
    \par To do so, we explore a natural connection between Graham's problem and the following very natural question attributed to Schrijver. Given a $d$-regular graph $G$ properly edge-coloured with $d$ colours, is it always possible to find a rainbow path with $d-1$ edges? We settle this question asymptotically by showing one can find a rainbow path of length $d-o(d)$. While this has immediate applications to Graham's question for example when $\Gamma=\mathbb{F}_2^k$, our general result above requires a more involved result we obtain for the natural directed analogue of Schrijver's question.
\end{abstract}

\section{Introduction}
In this paper, we study the following natural question in combinatorial number theory, first raised by Graham \cite{graham1971sums} in 1971, and reiterated by Erd\H{o}s and Graham in 1980~\cite{ErdosGraham}.
\begin{conj}[Graham, 1971]\label{conj:graham}For any $p$ prime and $a_1, a_2, \dots, a_d$ non-zero distinct elements of $\mathbb{Z}_p$, there exists a \emph{rearrangement} of the elements as $a_{i_1}, a_{i_2}, \dots, a_{i_d}$ such that all partial sums $\sum_{j=1}^t{a_{i_j}}$, $1\leq t \leq d$ are distinct. 
\end{conj}
In his original paper, Graham \cite{graham1971sums} poses the rearrangement problem in conjunction with several other related problems concerning sumset structures, that is, the structure of the set of all numbers that can be formed by taking sums of elements from a given subset, which can be considered the principal domain of additive combinatorics. However, later papers of Graham \cite{chung2007universal, buhler1994juggling, graham2013juggling} suggest that the rearrangement problem was at least partially motivated by practical applications to juggling.
\par The goal of the present paper is to provide asymptotic solutions to Graham's conjecture and several other related problems. To achieve this, we primarily use the lens of extremal and probabilistic combinatorics; in particular, we exploit a connection between Graham's conjecture and the rich area of finding \textit{rainbow subgraphs}. An
edge colouring of an undirected graph is \emph{proper} if no two edges sharing a vertex have the same colour. In
the directed setting, no pair of edges with a common start-vertex and no pair of edges with a common end
point may be monochromatic. In both cases, a subgraph of a coloured graph or digraph is \emph{rainbow} if all its edges
have distinct colours.  Given a subset $S$ of a group $\Gamma$, we define the \emph{(coloured) Cayley graph} on $\Gamma$ with \emph{generator set} $S$ to be the edge-coloured directed graph $\mathrm{Cay}(\Gamma,S)$ on vertex set $\Gamma$ with an edge from $a$ to $ag$ of colour $g$ for every $a \in \Gamma$ and $g \in S$\footnote{Here and throughout the paper, we use additive notation if the underlying group is assumed to be Abelian, and multiplicative notation otherwise.}. If $S = \{a_{i_1}, \ldots, a_{i_d}\}$ is a rearrangement of $S \subseteq \mathbb Z_p \setminus \{0\}$ with distinct partial sums, we see that $(a_{i_1}, a_{i_1}+a_{i_2}, \ldots, a_{i_1}+ \cdots +a_{i_d})$ is a rainbow directed path in $\mathrm{Cay}(\mathbb Z_p, S)$ with $d-1$ edges. Conversely, any rainbow directed path in $\mathrm{Cay}(\mathbb Z_p, S)$ with $d-1$ edges\footnote{Henceforth, in the context of directed graphs, a path always refers to a directed path, and in both directed and undirected settings, the \emph{length} of a path refers to the number of edges it contains.} is of the form $(x + a_{i_1}, x + a_{i_1} + a_{i_2}, \ldots, x + a_{i_1} + \cdots + a_{i_d})$ for some $x \in \mathbb Z_p$ and rearrangement $S = \{a_{i_1}, \ldots, a_{i_d}\}$ with distinct partial sums. Before saying more about our methods, we give a survey of what is known about Graham's conjecture as well as various generalisations thereof. 

Most of the progress towards Graham's conjecture has been in the cases when the \emph{generator set} $S=\{a_1, a_2, \dots, a_d\}$ is very small or very large. Indeed, summarising the work of many, an approach based on the Combinatorial Nullstellensatz~\cite{alon1999combinatorial} verifies the conjecture when $|S|\leq 12$ and direct, constructive arguments can be used when $|S|\geq p-3$, see~\cite{costa2020some, hicks2019distinct} and references therein. It also follows from the results of~\cite{muyesser2022random} (see Section 6.2) that Graham's conjecture is true for sets $S$ with $|S|=(1-o(1))p$. Recently, Kravitz~\cite{Kravitz} and independently Sawin~\cite{Sawin2015} showed that Graham's conjecture holds when $|S|\leq \log p/\log \log p$. In a subsequent paper, Bedert and Kravitz~\cite{BederdKravitz} significantly improved this bound by showing that the conjecture holds when $|S|\leq e^{(\log{p})^{1/4}}$, which has been an important milestone as this bound overcomes a natural barrier for the rectification techniques used in~\cite{Kravitz, Sawin2015}. 

Given an arbitrary group $\Gamma$ and some subset $S = \{a_1, \ldots, a_d\}\subseteq \Gamma$, we say $S$ is \emph{rearrangeable} if it is possible  to order the elements of $S$ as $a_{i_1}, a_{i_2}, \dots, a_{i_d}$ such that all partial products $\prod_{j=1}^t{a_{i_j}}$, $1\leq t \leq d$ are distinct. The observation mentioned earlier, which connects rearrangeable sets in $\mathbb Z_p$ to rainbow paths in Cayley graphs, extends to this more general context. That is, $S \subseteq \Gamma$ is rearrangeable if and only if $\mathrm{Cay}(\Gamma,S)$ contains a rainbow path of length $|S|-1$. Graham's conjecture states that all subsets of $\mathbb{Z}_p\setminus \{0\}$ are rearrangeable, but in fact, the study of rearrangeable sets has a rich history that predates Graham's conjecture. In 1961, Gordon~\cite{gordon1961sequences} introduced the notion of \emph{sequenceable groups} as those groups $\Gamma$ for which the set $S=\Gamma$ is rearrangeable.
 In the same paper, he characterised all sequenceable Abelian groups. A conjecture by Keedwell~\cite{keedwell1981sequenceable} from 1981 says that the only non-Abelian non-sequenceable groups are the dihedral groups of order $6$, $8$ and the quaternion group. This has recently been confirmed for all sufficiently large groups by Müyesser and Pokrovskiy~\cite{muyesser2022random}. In the same spirit, Ringel~\cite{ringeloldproblem}, motivated by some constructions arising in his celebrated proof of the Heawood map colouring conjecture \cite{ringel2012map}, raised the question of which groups can be ordered as $a_1,a_2,\ldots, a_{n}$ where $a_1$ is the identity, the partial products $a_1$, $a_1a_2$, $\ldots$ , $a_1a_2\cdots a_{n-1}$ are all distinct, and $a_1a_2\cdots a_{n}=e$; that is, for which groups $\Gamma$ does $\mathrm{Cay}(\Gamma, \Gamma)$ contain a rainbow cycle of length $|\Gamma| - 1$? This latter problem has similarly been resolved for large groups in~\cite{muyesser2022random}. The motivation for these kinds of problems comes from combinatorial design theory~\cite{evans2018orthogonal}. For example, a \emph{Latin square} is an $n \times n$ array filled with $n$ symbols such that each symbol appears exactly once in each row and once in each column. A Latin square is called \emph{complete} if every pair of distinct symbols appears exactly once in each order in adjacent horizontal cells and exactly once in adjacent vertical cells. It is not hard to see that any sequenceable group admits a Cayley (multiplication) table which is a complete Latin square. This connection ties sequenceability to decompositions of directed graphs into Hamiltonian paths~\cite{ollis2002sequenceable}, Heffter arrays~\cite{pasotti2022survey} and even to experimental designs~\cite{BATE2008336}.

Showing rearrangeability of arbitrary subsets of any group turns out to be significantly more difficult, and here there has only been limited progress, even in the simplest case of $\Gamma=\mathbb{Z}_p$, as in the setting of Graham's conjecture. We remark that a generalised version of Graham's conjecture for all cyclic groups was posed by Archdeacon, Dinitz, Mattern, and Stinson~\cite{archdeacon2015partial}. A slightly stronger conjecture is due to Alspach~\cite{BodeHarborth,costa2022sequences} who conjectured that if in Conjecture~\ref{conj:graham} we have an additional assumption that $\sum_{i=1}^d{a_i}\neq 0$, then all the partial sums are not only distinct but also non-zero. This is equivalent to saying that
for any subset $S \subseteq \mathbb Z_n \setminus \{0\}$, $\mathrm{Cay}(\mathbb Z_n,S)$ contains either a rainbow path of length $|S|$ or a rainbow cycle of length $|S|$. Alspach was motivated by applications to cycle decompositions of complete graphs~\cite{ALSPACH200177}, complete graphs plus or minus a $1$-factor~\cite{vsajna2002cycle,vsajna2003decomposition} and complete symmetric digraphs~\cite{ALSPACH2003165}. Various versions of Alspach's conjecture have been reiterated and stated, for example, for general Abelian groups by Costa, Morini, Pasotti and Pellegrini~\cite{costa2018problem}, and by Costa, Della Fiore, Ollis and Rovner-Frydman~\cite{costa2022sequences}. Below we state a version of Graham's conjecture for general Abelian groups, which is due to Alspach and Liversidge \cite{alspach2020strongly}. 

\begin{conj}\label{conj:Alspach}For any finite Abelian group $\Gamma$, any subset $S\subseteq \Gamma\setminus\{0\}$ is rearrangeable.
\end{conj}

Conjecture~\ref{conj:Alspach} is known to hold when $|S|\leq 9$~\cite{alspach2020strongly}; interestingly, the proof uses a method based on posets, which is distinct from the other methods mentioned above for the $\Gamma=\mathbb Z_p$ case. Our first result gives an approximate answer to these conjectures, and moreover, 
it applies to non-Abelian groups as well.

\begin{theorem}\label{thm:weakasymptotic-intro}
    For any finite group $\Gamma$ and any subset $S\subseteq \Gamma $ there exists an ordering of elements of $S$ in which at least $(1-o(1))|S|$ many partial products are distinct. \footnote{We note that here and in Theorem~\ref{thm:summary}(a) the asymptotic is w.r.t.\ $|S|$ only, and not w.r.t.\ the size of the ambient group $\Gamma$.}
\end{theorem}

 Theorem~\ref{thm:weakasymptotic-intro} guarantees that all subsets are approximately rearrangeable in a very general setup. However, there is yet another natural notion of an approximate rearrangeability introduced by Archdeacon, Dinitz, Mattern, and Stinson (see Problem 1 in~\cite{archdeacon2015partial}), where given $S$, we search for a subset $S'\subseteq S$ as large as possible such that $S'$ is rearrangeable. A greedy argument easily produces such an $S'$ of size $|S|/2$ (see the work of Hicks, Ollis, and Schmitt \cite{hicks2019distinct} for slightly better bounds). However, a major milestone here would be to find a rearrangeable subset $S'$ of size $(1-o(1))|S|$. Observe that the existence of such $S'$ implies that $S$ can be rearranged approximately in the sense of Theorem~\ref{thm:weakasymptotic-intro}, simply by arbitrarily permuting $S\setminus S'$. Moreover, it is feasible that such a strong approximate result could pave the way for exact results. Indeed, the aforementioned result of Müyesser and Pokrovskiy~\cite{muyesser2022random} showing that large non-Abelian groups are sequenceable relies on a strong approximation which holds in the regime when $\Gamma$ and $S$ have comparable size. More broadly, Keevash's \cite{keevash2014existence} celebrated proof of the existence of designs, the recent proof of Ringel's conjecture by Montgomery, Pokrovskiy, and Sudakov~\cite{ringel} (see also the independent proof of Keevash and Staden~\cite{keevash-staden-ringel}), and the resolution of the Ryser-Brualdi-Stein conjecture by Montgomery~\cite{Montgomery2024} 
 all rely on the approximate versions of the corresponding results ~\cite{rodl1985packing,approximateringel, keevash2022new} and use a similar framework, commonly referred to as \textit{the absorption method}, to obtain exact results from approximate ones.
 
\par We are able to achieve this strong approximation in three distinct settings, as summarised below.

\begin{theorem}\label{thm:summary}
    For any group $\Gamma$ and any subset $S\subseteq \Gamma$ of size $d := |S|$, there exists a rearrangeable subset $S' \subseteq S$ of size $(1-o(1))d$ if any one of the following holds. 
    \begin{enumerate}
        \item [(a)] \label{thm:strongasymptoticZ2-intro} $S$ contains only \emph{involutions}, that is, elements of order two. 
        \item [(b)] \label{thm:strongasymptotic-intro} $d=\Omega(|\Gamma|)$. 
        \item [(c)] \label{thm:Graham} $d \geq p^{3/4+o(1)}$ and $\Gamma=\mathbb{Z}_p$ for some prime $p$.
    \end{enumerate}
    
\end{theorem}

Note that Theorem~\ref{thm:summary}(a) implies that the same conclusion holds for any subset of the group $\mathbb{F}_2^k$. To prove Theorem~\ref{thm:weakasymptotic-intro} and Theorem~\ref{thm:summary}, we use various techniques from the toolkit of probabilistic combinatorics, such as robust expansion, and in the case of Theorem~\ref{thm:summary}(c), some tools from additive combinatorics, such as Pollard's inequality, which is a strengthening of the classical Cauchy-Davenport inequality. Theorem~\ref{thm:summary}(a) and Theorem~\ref{thm:summary}(b) are of special interest, as they are both consequences of much more general results concerning the existence of rainbow paths in properly edge-coloured graphs and digraphs. Questions of this nature are both natural and have been the subject of considerable research, going all the way back to Euler's work in 18th century on the existence of transversals in Latin squares. In the last decade alone, there have been several breakthroughs in this area. Most notably, the famous Ryser-Brualdi-Stein conjecture~\cite{ryser1967neuere, brualdi1991combinatorial, stein1975transversals} from 1967 states that in any Latin square of order $n$ there exists a \emph{transversal} of order $n-1$, that is $n-1$ cells which share no column, row or a symbol. This has equivalent formulations in terms of rainbow structures in graphs and digraphs. In the undirected setting, a Latin square of order $n$ corresponds to a proper edge colouring of the complete bipartite graph $K_{n,n}$ with $n$ colours, and a transversal corresponds to a rainbow matching. Last year, following a recent improvement on the best lower bound on the size of a rainbow matching of the form $n-o(n)$ by Keevash, Pokrovskiy, Sudakov and Yepremyan~\cite{keevash2022new}, Montgomery~\cite{Montgomery2024} resolved the conjecture in his tour de force work. In the directed setting, a Latin square of order $n$ corresponds to a proper edge colouring with $n$ colours of the complete symmetric digraph on $n$ vertices with loops at each vertex, and a transversal corresponds to a rainbow subgraph whose components are directed paths and cycles (see \cite{benzing2020long} for a more detailed explanation). 
There are also several related problems which inquire about the existence of rainbow subgraphs with a more specific structure; for example, subgraphs which are large path forests or long cycles (see \cite{gyarfas2014rainbow, benzing2020long}) or cycle factors of a particular cycle type (see \cite{friedlander1981partitions, AlpSolo}). We refer the reader to comprehensive surveys by Pokrovskiy~\cite{alexey-survey}, by Montgomery~\cite{Montgomery_2024}, and by Sudakov~\cite{sudakov2024restricted} for a broader overview of the area.  For the rest of this paper, we will only focus on the existence of rainbow paths in properly edge-coloured graphs and digraphs.
\par In this direction, Hahn~\cite{hahn} conjectured in 1980 that in any proper edge colouring of $K_n$ there exists a rainbow path of length $n-1$. This was refuted by Maamoun and Meyniel~\cite{maamoun1984problem}, by considering $\mathrm{Cay}(\mathbb{F}_2^k, \mathbb{F}_2^k\setminus\{0\})$ which has no rainbow Hamilton path.
Andersen's conjecture~\cite{andersen1989hamilton} from 1989 proposes a weakening of Hahn's conjecture and states that there exists a rainbow path of length $n-2$ in any properly coloured $K_n$. After a lot of partial progress, Andersen's conjecture has been proven asymptotically by Alon, Pokrovskiy, and Sudakov \cite{alon2017random} who showed the existence of rainbow paths of length $n-o(n)$ (see the work of Balogh and Molla \cite{BALOGH2019140} for the current best lower bound).
\par Schrijver~\cite{schrijver} asked for a far reaching generalisation of Andersen's conjecture by postulating the existence of a rainbow path of length $d-1$ in any properly $d$-edge-coloured $d$-regular graph $G$, and he verified this conjecture whenever $d\leq 10$. The best general bound on Schrijver's problem guarantees a path of length roughly $2d/3$, due to Chen and Li (unpublished, see \cite{babu2015heterochromatic}), and independently, Johnston, Palmer, and Sarkar~\cite{johnston2016rainbow}. Babu, Chandran, and Rajendraprasad~\cite{babu2015heterochromatic} showed that  if the graph is $C_4$-free then a rainbow path of length $d-o(d)$ exists, and moreover if the girth is of order $\Omega(\log{d})$, then a rainbow path of length $d-2$ exists. In work independent from ours, Conlon and Haenni~\cite{ConlonHaenni} have asymptotically resolved Schrijver's problem for random $d$-regular graphs. We resolve Schrijver's problem asymptotically in full generality, by showing the following. 

\begin{theorem}\label{thm:schrijver-asymptotic-intro} Any properly edge-coloured $d$-regular graph contains a rainbow path of length $(1-o(1))d$.
\end{theorem}
Observe that the $d=n-1$ case of Theorem~\ref{thm:schrijver-asymptotic-intro} corresponds to Andersen's conjecture. Thus, our result recovers the aforementioned asymptotic result of Alon, Pokrovskiy, and Sudakov \cite{alon2017random} for complete graphs and extends it into a much more difficult setting where the host graph can be extremely sparse. Furthermore,  Theorem~\ref{thm:schrijver-asymptotic-intro} also directly implies Theorem~\ref{thm:summary}(a), thereby giving a strong approximation for Conjecture~\ref{conj:Alspach} in the case of $\Gamma =\mathbb{F}_2^k$. 

Given the connection between rearrangements and rainbow directed paths described above, it is natural to wonder if the algebraic structure of a coloured Cayley graph is relevant for Graham's conjecture, or more generally, for Conjecture~\ref{conj:Alspach}. Specifically, we ask the following, which may be considered to be a directed generalisation of Schrijver's problem to \emph{directed $d$-regular digraphs}, that is, digraphs in which every vertex has in-degree and out-degree exactly $d$. 

\begin{problem}\label{problem:directed} Let $G$ be a $d$-regular digraph properly edge-coloured with $d$ colours. Does $G$ contain a directed rainbow path with $d-1$ edges?
\end{problem} 

This is known to be true asymptotically for complete symmetric digraphs by results of Benzing, Pokrovskiy and Sudakov~\cite{benzing2020long}, mirroring the asymptotic results on Andersen's conjecture~\cite{alon2017random}. An affirmative answer to
Problem~\ref{problem:directed} would be quite consequential, as this would resolve Graham's conjecture, Conjecture~\ref{conj:Alspach}, and the natural generalisation of Conjecture~\ref{conj:Alspach} to non-Abelian groups as well, and therefore even an approximate answer is of great interest. We are able to give such an approximate answer to Problem~\ref{problem:directed} in the dense regime. 

\begin{theorem}\label{thm:regulardigraph-intro}Let $G$ be a $d$-regular, properly edge-coloured digraph on $n$ vertices, with $d=\Omega(n)$. Then, $G$ contains a rainbow path with $(1-o(1))d$ edges.
\end{theorem}
Note that the above implies Theorem~\ref{thm:summary}(b) directly. A common theme in the proofs of Theorem~\ref{thm:weakasymptotic-intro} and Theorem~\ref{thm:schrijver-asymptotic-intro} is a reduction of the general case to the dense regime, hence Theorem~\ref{thm:regulardigraph-intro} plays a key role in our arguments. Of course, some graphs do not contain any dense spots and this translates to a strong expansion property that allows us to build a long rainbow path by an elementary argument, see \Cref{lem:warmup} for illustration. The challenge, then, lies in graphs that are neither dense, nor strongly expanding. This points to a recurring difficulty in our arguments, namely, we have to work with mild forms of expansion whilst building long rainbow paths. To see an example of how we achieve this, we refer the reader to Lemma~\ref{lem:mop argument} which introduces a novel procedure that outputs a rainbow path of asymptotically the optimal length in graphs that satisfy a rather delicate notion of expansion.

\textbf{Organisation.} 
We introduce some notation and state a few standard results and observations in Section~\ref{sec:notation}. 
Section~\ref{sec:pathforests} proves some versions of our main results where paths are replaced with path forests with few components. Section~\ref{sec:expansion} introduces some notions of expansion and builds a toolkit for working with expanders. Section~\ref{sec:dense} proves Theorem~\ref{thm:regulardigraph-intro} and \Cref{thm:summary}(c) using the tools derived in Sections~\ref{sec:pathforests} and~\ref{sec:expansion}. Section~\ref{sec:mop} contains a proof of Theorem~\ref{thm:schrijver-asymptotic-intro}. Section~\ref{sec:rearrangement} contains a proof of Theorem~\ref{thm:weakasymptotic-intro}. In Section~\ref{sec:concluding} we make some concluding remarks and suggest some directions for future research.

\section{Notation and preliminaries}\label{sec:notation}
\textbf{Notation.} 
The digraphs we consider are loopless, and for each pair $(u,v)$ of distinct vertices, we allow at most one edge from $u$ to $v$, which we denote by $(u,v)$. We do, however, allow both edges $(u,v)$ and $(v,u)$ to appear in the same digraph. If $G$ is a (possibly edge-coloured) digraph, then for $U, V \subseteq V(G)$, we write $e_G(U,V)$ to denote the number of edges $(u,v)$ with $u \in U$ and $v \in V$. As special cases, for a vertex $v \in V(G)$, we denote the \emph{out-degree} of $v$ by $\deg^+_G(v) := e_G(\{v\}, V(G))$ and the \emph{in-degree} of $v$ by $\deg^-_G(v) := e_G(V(G), \{v\})$.  We also write $\deg_G^+(v,U) := e_G(\{v\},U)$, $\deg_G^-(v,U) := e_G(U,\{v\})$, $\partial^+_G(U) := e_G(U, V(G) \setminus U)$, and $\partial^-_G(U) := e_G(V(G) \setminus U, U)$. We omit the subscript $G$ when the digraph $G$ is clear from context.

In any of these instances, when we wish to count only those edges whose colours come from a particular colour set $C$, we insert the symbols `$;C$' before closing the parentheses (for instance, we write $\deg^+(v,U;C)$ for the number of edges from $v$ to $U$ with colours from $C$). We denote the minimum out-degree and in-degree of $G$ by $\delta^+(G)$ and $\delta^-(G)$, respectively, and we write $\delta^{\pm}(G) := \min \{\delta^+(G), \delta^-(G)\}$ for the \emph{minimum semi-degree} of $G$. If $H$ is a subgraph of $G$, we use $C(H)$ to denote the set of colours that appear in $E(H)$. Similarly, for an edge $(u,v) \in E(G)$, we use $C(u,v)$ to denote the colour of $(u,v)$. 
Also, for a vertex subset $V \subseteq V(G)$ and a colour subset $C \subseteq C(G)$, we use $G[V;C]$ to denote the subgraph induced by $V$ and $C$.

A digraph is called \emph{symmetric} if every edge is contained in a $2$-cycle. To an undirected simple graph $G$, we associate a symmetric digraph $\hat G$ by replacing each undirected edge $uv \in E(G)$ with the two directed edges $(u,v)$ and $(v,u)$. If $G$ is edge-coloured, then both $(u,v)$ and $(v,u)$ are given the colour of $uv$. The resulting edge colouring of $\hat G$ is proper if and only if the edge colouring of $G$ is proper. Further, the rainbow paths (and walks) in $\hat G$ are in one-to-one correspondence with the rainbow paths (and walks) in $G$. Thus, for our purposes, we can treat $G$ and $\hat G$ as the same object when it is convenient; that is, we consider undirected graphs as special cases of directed graphs. We use the same notation for undirected graphs as for directed graphs described above, except that we omit the superscript `$+$' or `$-$' (for example, for a vertex $v \in V(G)$, we write $\deg_G(v) := \deg^+_{\hat G}(v) = \deg^-_{\hat G}(v)$ for the \emph{degree} of $v$).

We use ``$\alpha \ll \beta$ implies $P(\alpha, \beta)$'' as shorthand to denote that there exists an increasing function $f$ such that for any $\beta$, $P(\alpha, \beta)$ holds for $\alpha \leq f(\beta)$.

\begin{lemma}[Chernoff's inequality]\label{chernoff}
    Let $X$ be a sum of independent Bernoulli random variables with $\mathbb E(X) = \mu$. Then for every $t > 0$,
    \begin{itemize}
        \item $\mathbb P(X \leq \mu - t) \leq \exp(-t^2/(2\mu))$;
        \item $\mathbb P(X \geq \mu + t) \leq \exp(-t^2/(2\mu + t))$.
    \end{itemize}
\end{lemma}

We conclude the section with two simple observations, which extend classical arguments from the non-rainbow setting for finding long paths.

\begin{obs}
    \label{lem:GreedyPathMinDegree}
    A properly edge-coloured digraph with minimum out-degree $d$ contains a rainbow path of length at least $d/2$.
\end{obs}
\begin{proof}
    Let $P$ be a rainbow path of maximum length $\ell$, and let $v$ be the terminal vertex of $P$. Let $C'$ be the set of colours not in $C(P)$. By maximality of $P$, every $C'$-coloured out-edge from $v$ must go to $V(P) \setminus \{v\}$, so $d-\ell \leq \deg(v;C') \leq \ell$. Thus $\ell \geq d/2$.
\end{proof}

\begin{obs}
    \label{lem:GreedyOnePath}
  A properly edge-coloured undirected graph with average degree $d$ contains a rainbow path of length at least $d/4$.
\end{obs}
     \begin{proof} Iteratively delete all vertices with degree less than $d/2$ as long as possible. The number of edges deleted is less than $dn/2$, so the remaining graph is non-empty and by construction has minimum degree at least $d/2$. Now applying \Cref{lem:GreedyPathMinDegree} to the associated symmetric digraph gives a rainbow path of length at least $d/4$.
     \end{proof}

\section{Rainbow path forests with few components}\label{sec:pathforests}
The goal of this section is to prove some versions of our main results where paths are replaced with path forests with few components. A \emph{path forest} in a (directed) graph $G$ is a subgraph of $G$ whose components are paths. In this section, we prove two auxiliary lemmas on the existence of large rainbow path forests with few components in properly edge-coloured graphs and digraphs. The first one only applies to undirected graphs and requires the maximum degree to be sublinear in the number of vertices, but provides a rainbow forest with number of edges asymptotic in \emph{average} degree. The second lemma applies in both the directed and undirected case and it makes essentially no requirement on the degree, but it only provides a rainbow path forest with number of edges asymptotic in \emph{minimum} degree. This distinction will be important while passing to subgraphs where we might lose control of the minimum degree, but keep control of the average degree.

The following is the first of our rainbow path forest lemmas, which is proved by an iterative application of \Cref{lem:GreedyOnePath} in the undirected case.

    \begin{lemma}\label{lem:sparsepathforest}
        For any $\eps >0$ there exists $d_0$ such that for all $d\geq d_0$ the following holds.  If $G$ is a properly edge-coloured graph on $n$ vertices with average degree $d$ and maximum degree $\Delta \leq \eps n/4$, then $G$ contains a rainbow path forest with at most $\left\lceil \log_{8/7}{\eps^{-1}}\right\rceil$ many components and at least $(1-\eps)d$ edges.
     \end{lemma}

\begin{proof}
Let $d_0:=\eps^{-1} \log_{8/7} \eps^{-1}$. 
We build our desired rainbow path forest iteratively, where the first path forest $\mathcal P_1$ is a longest rainbow path in $G$, and given the current path forest $\mathcal P_i$, we obtain $\mathcal P_{i+1}$ by simply incorporating a longest rainbow path to $\mathcal P_i$ that is vertex-disjoint and colour-disjoint from $\mathcal P_i$. We stop as soon as we reach $\mathcal P_m$ with at least $(1-\eps)d$ edges or with $m = \left \lceil \log_{8/7} \eps^{-1} \right \rceil$. %It remains to show that this process produces a rainbow path forest with the desired properties. 

At the end of this process, we will have produced rainbow path forests $\mathcal{P}_1 \subseteq \ldots \subseteq \mathcal{P}_m$ %such that each $\mathcal P_i$ has at most $i$ components. %where each $\mathcal{P}_i$ has $i$ components and $m \le \left \lceil \log_{8/7} \eps^{-1} \right \rceil.$
%In particular, $\mathcal P_m$ has 
with at most $m \leq \left \lceil \log_{8/7} \eps^{-1} \right \rceil$ components, so it remains to show that $e(\mathcal P_m) \geq (1-\eps)d$. %, so we may assume that $e(\mathcal{P}_m) < (1-\eps)d$, as otherwise we are done.
We may assume that the process stopped because $m = \left \lceil \log_{8/7} \eps^{-1} \right \rceil$, as otherwise we are done already. 

Now, it is enough to show that $e(\mathcal{P}_i) \ge (1 - (7/8)^i)d$ for each $1 \le i \le m$, as this gives $e(\mathcal P_m) \geq \left(1 - (7/8)^m\right)d \geq (1 - \eps)d$. We do this by induction on $i$. The base case holds since by \Cref{lem:GreedyOnePath}, we have $e(\mathcal{P}_1)\ge d/4 \geq (1-(7/8))d$. Suppose now that $e(\mathcal{P}_{i-1}) \ge (1 - (7/8)^{i-1})d$ for some $2 \le i \le m$. Since the process did not stop at $\mathcal P_{i-1}$, we also know that $e(\mathcal P_{i-1}) < (1-\eps)d$.

We will next show that the graph $G'$ obtained from $G$ by removing the vertices and colours of $\mathcal{P}_{i-1}$ still has large enough average degree so that by another application of \Cref{lem:GreedyOnePath}, the new path we add to $\mathcal{P}_{i-1}$ sufficiently increases its size. To this end, note first that 
\[|V(\mathcal{P}_{i-1})| \leq e(\mathcal{P}_{i-1}) + i-1 \leq (1-\eps)d + \log_{8/7} \eps^{-1} \leq d\]
since $d \ge d_0 \ge \eps^{-1} \log_{8/7} \eps^{-1}$.
Let $d' := d - e(\mathcal{P}_{i-1})$, which satisfies $\eps d \leq d' \leq (7/8)^{i-1} d$. Since $G$ is properly edge-coloured, $G \setminus C(\mathcal{P}_{i-1})$ has at least $nd/2-|C(\mathcal{P}_{i-1})|n/2 =nd'/2$ edges, and at most $\Delta|V(\mathcal{P}_{i-1})|$ of these are incident to some vertex in $V(\mathcal{P}_{i-1})$. This leaves at least
\[nd'/2 - \Delta|V(\mathcal{P}_{i-1})| \geq nd'/2 - \eps n d/4 \geq nd'/4\]
edges in $G'$ since $\Delta \leq \eps n/4$ and $d' \geq \eps d$. Consequently, $G'$ has average degree at least $d'/2$, so by \Cref{lem:GreedyOnePath}, $G'$ has a rainbow path of length at least $d'/8$. So,
\[e(\mathcal{P}_{i})\ge e(\mathcal{P}_{i-1}) + \frac{d'}{8} = d - \frac{7d'}{8} \geq \left(1 - \left(\frac{7}{8}\right)^{i}\right)d,\]
which completes the induction step, and hence the proof.
\end{proof} 

We proceed to our second rainbow path forest lemma. The proof of this lemma is motivated by ideas from \cite{benzing2020long}.

\begin{lemma}\label{lem:rainbow path forest}
	Let $0 < \eps \leq 1$, and let $G$ be a properly edge-coloured digraph with minimum in-degree $\delta^-(G) \geq 9\eps^{-3}$. Then $G$ contains a rainbow path forest with at most $9\eps^{-2}$ components and at least $(1 - \eps)\delta^-(G)$ edges.
\end{lemma}

\begin{proof}
	Set $q = s := \lceil 2\eps^{-1}\rceil \leq 3\eps^{-1}$, and set $d := \delta^-(G)$. Let $\mathcal P$ be a rainbow path forest with at most $qs$ components, with the maximum possible number of edges. Assume for the sake of contradiction that $e(\mathcal P) < (1-\eps)d$. Then since $d = \delta^-(G) \geq 9\eps^{-3} \geq \eps^{-1}qs$, we have
	\begin{equation}\label{eq:vertices in path forest}
	    |V(\mathcal P)| < (1 - \eps)d + qs \leq d \leq |V(G)|,
	\end{equation}
	so by adding isolated vertices to $\mathcal P$ one by one if necessary, we may assume that $\mathcal P$ has exactly $qs$ components. We call the vertices of in-degree $0$ in $\mathcal P$ \emph{start-vertices} and the vertices of out-degree $0$ in $\mathcal P$ \emph{end-vertices}.

    Our general strategy here is to show that we can increase the size of our rainbow path forest (without increasing the number of components), and hence obtain a contradiction to maximality. Here, the easiest way to extend the family would be to find a start-vertex $u$ and an in-neighbour $v$ of $u$ which is not in $\mathcal P$ such that the colour of $(v,u)$ is not in $C(\mathcal P)$ (we would also be happy if $v$ was an end-vertex from a path of $\mathcal P$ not containing $u$). % and vertex of $G$ not in $\mathcal{P}$ which has an out-neighbour using a colour not in $C(\mathcal{P})$ among the start-vertices of paths in $\mathcal{P}$.
    Unfortunately, all of the in-neighbours of the start-vertices using colours not in $C(\mathcal{P})$ might already be in $\mathcal P$ and not be end-vertices. If that is the case, we can still use these in-neighbours in the following way. Suppose $u$ is a start-vertex of a path in $\mathcal P$ and $v$ is its in-neighbour using a colour not in $C(\mathcal{P})$. Then we can replace the out-edge of $v$ in $\mathcal P$ with the edge $(v,u)$ to obtain a new path forest with the same number of components.
    % We can move $v$, together with the part of the path it belonged to in $\mathcal P$ ending with it, and prepend it to the path starting with $u$ using the new colour edge between them.
    This ``switch'' has the effect of freeing up the colour of the out-edge of $v$ in $\mathcal{P}$,
    % (if $v$ is an end-vertex of a path in $\mathcal P$, then we simply increase the size of the path forest and win). 
    % So by doing this type of ``switch'' we can free up some additional colours,
    which can now itself be used in new switches. We will show that this will eventually lead to a short sequence of switches which does allow us to increase the size of the forest.
    
    In order to carry out this strategy, let us introduce some notation. 
    Let $U$ be the set of start-vertices, and let $W$ be the set of end-vertices.
    For $v \in V(\mathcal P) \setminus W$, let $C(v)$ denote the colour of the edge coming out of $v$ in $\mathcal P$ (i.e., the colour we will be able to free up if we can perform a switch at $v$).
	
	Let $\{Q_0, \ldots, Q_{s-1}\}$ be a partition of $U$ into sets of size exactly $q$. This partition will allow us to avoid issues with chains of switches. 
    % , and in particular with the initial part of the chain changing the picture. So, 
    In particular, a switch at level $i$ will only ever use a ``fresh'' set of start-vertices from the set $Q_i$.
    
    Let $C_0 := C(G) \setminus C(\mathcal P)$ denote the set of colours in $G$ that do not appear in $\mathcal P$ (initial, ``free'' colours). For $1 \leq i \leq s$, let
	\begin{align*}
		V_i &:= \{v \in V(G) : \deg^+(v,Q_{i-1}; C_{i-1}) \geq 2\}; \\
        V_i' &:= \{v \in V(G) : \deg^+(v,Q_{i-1}; C_{i-1})=1\}; \\
		C_i &:= C_{i-1} \cup \{C(v) : v \in V_i \cap V(\mathcal P) \setminus W\}.
	\end{align*}

Here, $C_i$ should be thought of as ``freeable'' colours at stage $i$ (meaning they can be made available after performing a chain of up to $i$ switches). $V_i$ can be thought of as the set of vertices $v$ which can be appended to a path with start-vertex in $Q_{i-1}$ (thereby freeing $C(v)$ which gets added to $C_i$) by using a colour which we managed to free up in the previous stage. A slight technicality here is that we actually require $v$ to have at least $2$ options of doing so, in order to be able to avoid creating a cycle if $v$ is only connected to the start-vertex of its own path. The set $V_i'$ consists only of vertices with a single such option.

The following claim formalises the idea that if any $V_i$ or $V_i'$ contains a vertex outside $V(\mathcal{P})$, we can incorporate this vertex to the rainbow path forest to increase its size via a chain of switches. The claim also says that we cannot have in $V_i$ a vertex from $W$, or we would be able to join two paths and hence also increase the size of the forest in the same way. 
% (having two options is important here to avoid closing a cycle rather than merging two paths).
See \Cref{fig:path-forest} for an illustration.

\begin{figure}[h]
    \centering
    \includegraphics[width=0.93\linewidth]{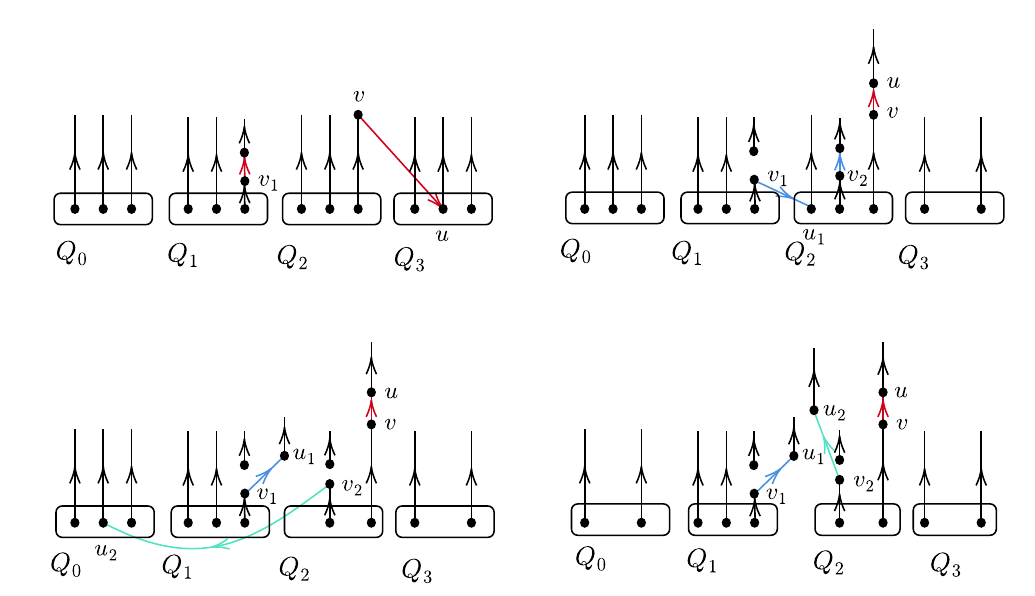}
    \caption{Illustration of the argument in the claim with $i=4$, $v \in W \cap V_4$, red is in $C_3 \setminus C_2$, $C(v_1)$ is red, $v_1 \in V_3$, blue is in $C_2$, $C(v_2)$ is blue, $v_2 \in V_1$, and green is in $C_0$.}
    \label{fig:path-forest}
\end{figure}
    
    \begin{claim*}
        For every $1 \leq i \leq s$, we have $V_i \subseteq V(\mathcal P) \setminus W$ and $V_i' \subseteq V(\mathcal P)$. 
    \end{claim*}
    \begin{cla_proof}
        Suppose not. Then there exists a vertex $v$ such that either $v \in W \cap V_i$ or $v \in (V_i \cup V_i') \setminus V(\mathcal P)$.
       
        By definition, since $v \in V_i \cup V_i'$ there exists an out-neighbour $u \in Q_{i-1}$ of $v$ such that $C(v,u) \in C_{i-1}$. In fact, if $v \in V(\mathcal{P})$, it is also in $V_i$, so there are at least two choices for such a $u$, and we can ensure that $u$ and $v$ are not part of the same path in $\mathcal{P}$. Now, if $v \in V(\mathcal{P})$, so also in $W$, we append the path of $\mathcal{P}$ ending in $v$ to the path of $\mathcal{P}$ starting with $u.$ If $v \notin V(\mathcal{P})$ we just append $v$ itself. Now, if $C(v,u) \in C_0$, we maintain rainbowness and manage to increase the size of our path forest, contradicting maximality. Otherwise, we are using this colour twice, so we remove the edge $(v_1,w_1)$, which already used colour $C(v,u)$ in  $\mathcal{P}$ to obtain a new rainbow path forest $\mathcal{P}_1$.  We note that $e(\mathcal{P}_1)=e(\mathcal{P})$, that there is exactly one edge of $\mathcal{P}$ not in $\mathcal{P}_1$, namely the one using colour $C(v,u)$, and that all of the vertices in $Q_{0} \cup \ldots \cup Q_{i-2}$ are still start-vertices of paths in $\mathcal{P}_1$. 

        Since $C(v,u) \notin C_0$, there exists $1\le i_1\le i-1$ such that $C(v,u) \in C_{i_1} \setminus C_{i_1-1}.$ Since $C(v_1)=C(v,u)$ we conclude $v_1 \in V_{i_1}$ and we can pick its out-neighbour $u_1 \in Q_{i_1-1}$, which is not on the same path in $\mathcal{P}_1$ as $v_1$, such that $C(v_1,u_1)\in C_{i_1-1}$ and we can repeat the argument above.   
        At stage $j$ of the process, we have an edge $(v_j,u_j)$ with 
        \begin{enumerate}
            \item \label{itm:switch1}
            $C(v_j,u_j)\in C_{i_j-1}$,
            \item \label{itm:switch2} $u_j \in Q_{i_j-1},$ and
            \item \label{itm:switch3} $v_j \in V_{i_j}\cap V(\mathcal{P}) \setminus W$,
        \end{enumerate}  
        where $1\le i_j <i_{j-1} < \ldots < i_{1}\le i-1,$ and a path forest $\mathcal{P}_{j}$ such that 
        \begin{enumerate}[label=\alph*)]
            \item \label{itm:path1}
            $e(\mathcal{P}_{j})=e(\mathcal{P})$, 
            \item \label{itm:path2} all of the edges of $\mathcal P$ with colours in $C_{i_j-1}$ are still in $\mathcal{P}_j$,
            \item \label{itm:path3} $Q_0\cup \ldots \cup Q_{i_j-1}$ are still start-vertices of paths in $\mathcal{P}_{j}$,
            \item \label{itm:path4} $v_j$ is an endpoint of a path in $\mathcal{P}_j,$ and $u_j$ is not on that same path. 
        \end{enumerate}
        
        Property \ref{itm:path4} allows us to add the edge $(v_j,u_j)$ to $\mathcal{P}_j$.
         If $C(v_j,u_j) \in C_0$ this maintains rainbowness and we increase the size, contradicting maximality. Otherwise, there exists $1\le i_{j+1}<i_j$ such that $C(v_j,u_j) \in C_{i_{j+1}} \setminus C_{i_{j+1}-1}.$ By definition, this implies there is a $v_{j+1} \in V_{i_{j+1}}\cap V(\mathcal{P}) \setminus W$ such that $C(v_{j+1}) \in C_{i_{j+1}}$ and the edge $(v_{j+1},w_{j+1})$ using this colour is in $\mathcal{P}_j$. We delete $(v_{j+1},w_{j+1})$ to create $\mathcal{P}_{j+1}$. We pick $u_{j+1} \in Q_{i_{j+1}-1}$ as an out-neighbour of $v_{j+1}$ which is not on the same path in $\mathcal{P}_{j+1}$, and has $C(v_{j+1},u_{j+1}) \in C_{i_{j+1}-1}$, which we can by definition of $V_{i_{j+1}} \ni v_{j+1}$. So, by construction, we maintain properties \ref{itm:switch1}--\ref{itm:switch3}.
         
         Let us now verify the properties of $\mathcal{P}_{j+1}$. We added one edge and deleted one edge, so property \ref{itm:path1} still holds. We only deleted one edge, which had a colour in $C_{i_{j+1}}$, so property \ref{itm:path2} is also still preserved. Only one vertex stopped being a start-vertex, namely $u_j \in Q_{i_j-1}$, so property \ref{itm:path3} is maintained. Since we deleted $(v_{j+1},w_{j+1})$ which was an edge in $\mathcal{P}_j$ we do make $v_{j+1}$ an endpoint of a path in $\mathcal{P}_{j+1}$, and we picked $u_{j+1}$ not to be on that path so property \ref{itm:path4} also holds.

         Since the indices $i_j$ strictly decrease, this shows we can repeat this process until at some point we increase the size of our path forest,
         % . Since, indices $i_j$ strictly decrease, this needs to happen at some point
         giving us a contradiction.
        \end{cla_proof}

Before moving on, we draw two immediate conclusions from the claim. First, $V_i \subseteq V(\mathcal{P}) \setminus W$ implies that for each $1 \leq i \leq s$, every vertex $v \in V_i$ witnesses a unique colour $C(v) \in C_i \cap C(\mathcal P)$, so
	\begin{equation}\label{eq:colour growth}
		|C_i \cap C(\mathcal P)| \geq |V_i|.
	\end{equation}
	Second, $V_i' \subseteq V(\mathcal P)$, together with~\eqref{eq:vertices in path forest}, implies that
 \begin{equation}\label{eq:sending one edge}
 |V_i'| \leq |V(\mathcal P)| \leq d.
 \end{equation}
 %which will be useful in arguing a certain lower bounding $|V_i|$ since we 

Our plan now is to argue that the number of freeable colours, namely $|C_i|$, grows so fast that $C_s$ must contain at least~$d$ colours from $C(\mathcal P)$, which is impossible. We will achieve this by counting all edges ending in $Q_{i-1}$ using colours from~$C_{i-1}$. There are many such edges by our minimum degree assumption, which will force $V_i$ to be large enough to ensure that many new colours are freed up in each stage by~\eqref{eq:colour growth}. 
% On one hand, we will get a lot of edges by our minimum degree assumption. On the other hand, the claim will allow us to argue that there are many new colours we free up in each stage. The following two bounds are how the claim is used; they tell us that to lower bound $|C_i|$, it suffices to lower bound $|V_i|$ and that $V_i'$ is not too large (so it does not contribute much in the above-mentioned count). 

	More specifically, we prove a lower bound on $|V_i| \leq |C_i \cap C(\mathcal P)|$ in terms of $|C_{i-1} \cap C(\mathcal P)|$ by counting
    the quantity $e(V(G), Q_{i-1}; C_{i-1})$ as follows. On the one hand, because $C(G) \setminus C(\mathcal P) = C_0 \subseteq C_{i-1}$, the number of colours in $C(G) \setminus C_{i-1}$ is only $|C(\mathcal P)| - |C_{i-1} \cap C(\mathcal P)| \leq (1-\eps)d - |C_{i-1} \cap C(\mathcal P)|$. Therefore, since $G$ is properly edge-coloured with minimum in-degree $d$, we have $\deg^-(u; C_{i-1}) \geq |C_{i-1} \cap C(\mathcal P)| + \eps d$ for every $u \in Q_{i-1}$, hence
    \[e(V(G), Q_{i-1}; C_{i-1}) \geq (|C_{i-1} \cap C(\mathcal P)| + \eps d)q.\]
    On the other hand, by \eqref{eq:sending one edge} and our choice of $q$, we have
	\begin{align*}
		(|C_{i-1} \cap C(\mathcal P)| + \eps d)q &\leq e(V(G), Q_{i-1}; C_{i-1}) \\
        &= e(V_i, Q_{i-1};C_{i-1}) + e(V_i', Q_{i-1}; C_{i-1}) \\
		&\leq |V_i|q + |V_i'| \\
        &\leq (|V_i| + \eps d/2)q.
	\end{align*}
	Rearranging and applying~\eqref{eq:colour growth} gives 
	\[|C_i \cap C(\mathcal P)| \geq |V_i| \geq |C_{i-1} \cap C(\mathcal P)| + \eps d/2,\]
	so by induction and our choice of $s$, we have
	\[|C_s \cap C(\mathcal P)| \geq s \eps d/2 \geq d > |C(\mathcal P)|,\]
    which is a contradiction.
	% But $C_s \cap C(\mathcal P) \subseteq C(\mathcal P)$, so it can only have size at most $(1-\eps)d$, a contradiction.
\end{proof}

\section{Expansion-based tools}\label{sec:expansion}

In this section, we establish a number of useful properties of a certain type of expander (directed) graphs. (These properties are stated for digraphs, but the corresponding properties for undirected graphs also hold, and they follow from the directed version by considering the associated symmetric digraph.) The motivation is that in the previous section, we saw how to find
% the rainbow path problem can be solved if we search for
a rainbow path forest with few components in properly edge-coloured (directed) graphs. In expander digraphs, we are able to connect the endpoints of the path forest into a single long path (see Lemma~\ref{lem:connecting lemma}). Furthermore, in \textit{any} dense regular digraph, we are able to find an expander subgraph that preserves all the essential properties of the original graph (see Lemma~\ref{lem:find robust expander}). These two lemmas are used frequently in the remainder of the paper.

\par We start by proving an introductory result in \Cref{sec:strong-expanders} assuming a stronger notion of expansion than we will use throughout the paper, as a warm-up to introduce and motivate some of our ideas. In \Cref{sec:robust-expanders}, we introduce the notion of robust expansion  and show how to find a robust expander in a nearly regular dense digraph. In \Cref{sec:connectivity}, we show that in properly coloured digraphs that are robust expanders, every pair of vertices remains connected via rainbow paths, even after the deletion of a certain number of vertices and colours. Both of these results are used in subsequent sections to handle the dense versions of the corresponding theorems.

\subsection{Long rainbow paths in  strong expanders}\label{sec:strong-expanders}

In this section, we show how to find long rainbow paths in digraphs with very strong expansion properties. The main lemma of this section (\Cref{lem:warmup}) is not used later in our paper, but we present it as a warm-up to highlight the connection between expansion and long rainbow paths. In particular, this lemma implies that an $n$-vertex random $d$-regular (directed) graph with $d = o(n)$ contains a rainbow path of length $d - o(d)$ with high probability. We learned that Conlon and Haenni \cite{ConlonHaenni} independently proved a similar lemma for this purpose in the undirected setting, using a P\'osa rotation argument. Our proof is slightly different, and it applies in both directed and undirected settings. Later on, we prove another lemma (\Cref{lem:long path in robust expander}) which implies the same conclusion for dense random $d$-regular (directed) graphs.

\begin{definition}
    A digraph $G$ is \emph{$(\eps,d)$-locally sparse} if every set $U \subseteq V(G)$ of size at most $d$ satisfies $e_G(U, U) \leq \eps d^2$.
\end{definition}

\begin{lemma}\label{lem:warmup}
    Let $\eps \in (0,1)$, and let $d \in \mathbb N$. Let $G$ be a properly edge-coloured digraph of minimum out-degree $d$. If $G$ is $(\eps^2/4, d)$-locally sparse, then $G$ has a rainbow directed path of length at least $(1-\eps)d$.
\end{lemma}
\begin{proof}
    Define an $(\ell,k)$-\emph{broom} in $G$ to be a rainbow subgraph $H = P \cup T$ of $G$, where $P$ is a path of length $\ell$ ending at some vertex $v$, and $T$ is an outward-oriented star with $k$ edges centred at $v$, with $V(P) \cap V(T) = \{v\}$. Clearly $G$ has a $(0, \lceil \eps d/2 \rceil)$-broom since $\delta^+(G) \geq d \geq \lceil \eps d/2 \rceil$.

    Consider an $(\ell,\lceil \eps d/2 \rceil)$-broom $H = P \cup T$ in $G$ with $\ell$ as large as possible. Suppose, towards a contradiction, that $G$ has no rainbow path of length at least $(1-\eps)d$, so $\ell < (1 - \eps)d - 1$. Let $v$ be the centre of $T$.
    For each vertex $u \in V(T) \setminus \{v\}$, there are at least $\eps d$ colours leaving $u$ which do not appear among $E(P) \cup \{(v,u)\}$. Let $S_u$ be this set of colours, and let
    \[N_u := \{w \in N^+(u) : C(u,w) \in S_u\}\]
    be the out-neighbourhood of $u$ in these colours. By the maximality of $H$, we have
    \[|N_u \setminus V(P)| < \eps d/2,\]
    for every $u \in V(T) \setminus \{v\}$, or else we would have an $(\ell+1, \lceil \eps d/2 \rceil)$-broom. This gives
    \[e(V(H), V(H)) \geq e(V(T) \setminus \{v\}, V(P)) \geq \sum_{u \in V(T) \setminus \{v\}} |N_u \cap V(P)| \geq \eps^2 d/4.\]
    But $|V(H)| < (1-\eps)d-1 + \lceil \eps d/2 \rceil \leq d$, so $G$ fails to be $(\eps^2/4, d)$-locally sparse, a contradiction.
\end{proof}

\subsection{Robust expanders and their existence in dense digraphs}\label{sec:robust-expanders}

In this section, we introduce the following robust expansion property, which was first given in this form in \cite{kuhn2010hamiltonian}, and was studied systematically in the work of K\"uhn, Lo, Osthus, and Staden \cite{kuhn2015robust}. %As we will see in Section~\ref{sec:finding}, this notion is intimately linked with the notion of cut-density, and so robust expansion relates to analogous notions studied by Conlon, Fox, and Sudakov~\cite{conlon2014cycle} and even before, by Koml\'os and Szemer\'edi \cite{komlos1996topological}.

\begin{definition}
    Let $G$ be a directed graph on $n$ vertices. For $U \subseteq V(G)$ and $\nu > 0$, we define the \emph{$\nu$-robust out-neighbourhood} of $U$ in $G$ to be the set
    \[RN_{\nu,G}^+(U) := \{v \in V(G) : |N^-(v) \cap U| \geq \nu n\}.\]
    We say that $G$ is a \emph{robust $(\nu, \tau)$-out-expander} if every $U \subseteq V(G)$ with $\tau n \leq |U| \leq (1-\tau)n$ satisfies
    \[|RN_{\nu, G}^+(U) \setminus U| \geq \nu n.\]
    Similarly, we define the \emph{$\nu$-robust in-neighbourhood} of $U$ in $G$ to be the set
    \[RN_{\nu,G}^-(U) := \{v \in V(G) : |N^+(v) \cap U| \geq \nu n\}.\]
    We say that $G$ is a \emph{robust $(\nu, \tau)$-in-expander} if every $U \subseteq V(G)$ with $\tau n \leq |U| \leq (1-\tau)n$ satisfies
    \[|RN_{\nu, G}^-(U) \setminus U| \geq \nu n.\]
\end{definition}

The following lemma shows how to find such a robust expander subgraph with essentially the same degrees in any
digraph with linear minimum degree and similar in- and out-degrees at every vertex. To give some intuition, robust expanders are essentially characterised by the lack of sparse cuts across linear-sized vertex sets; having no sparse cuts automatically gives a lower bound on the size of the robust neighbourhoods of sets by a simple counting argument. So as long as there exists a sparse cut $(V_1,V_2)$ in a graph $G$ with $V_1$ and $V_2$ not too small, we can delete all edges between $V_1$ and $V_2$, and repeat the same argument in parallel in $G[V_1]$ and $G[V_2]$. The algorithm eventually outputs a partition of $V(G)$ as $V_1,\ldots,V_k$ where there are no sparse cuts within any $V_i$, so that each $G[V_i]$ is a robust expander. Indeed, in \cite{kuhn2015robust} and \cite{gruslys2021cycle}, a formal argument in this direction is given to partition a given dense graph into its robustly expanding pieces whilst deleting only a negligible fraction of the edge-set.
\par Here we implement a similar algorithm for digraphs that produces a single robust expander instead of a partition, as this will be sufficient for our purposes. In an application in Section~\ref{sec:mop}, we will also require an additional technical property \ref{findrobustexpanderdense:3}. A careful analysis of the proof shows that the dependency between the parameters $\nu$ and $\alpha$ is double exponential when $\tau \leq \alpha$, while if we did not require \ref{findrobustexpanderdense:3}, one could get single exponential dependency. 

% in the first case $2^{-2^{1/\tau \log{1/\alpha}\log{1/\alpha}}$

%2^{1/\alpha}
%The proof is similar to previous arguments given in \cite{kuhn2015robust, gruslys2021cycle} and follows the above rough outline, but we give the details for the sake of completeness.

\begin{lemma}[Pass-to-expander lemma]\label{lem:find robust expander}
    For any $\delta, \tau, \alpha > 0$, there exist $\gamma, \nu > 0$ such that the following holds for every positive integer $n$ and every $n$-vertex digraph $G$. Suppose that $\delta^\pm(G) \geq \alpha n$ and that $|\deg^+(v) - \deg^-(v)| \leq \gamma n$ for every $v \in V(G)$. Let $w : V(G) \to \mathbb R_+$ be any weight function on the vertices of $G$. Then there exists a nonempty induced subgraph $G'$ of $G$ satisfying the following:
    \stepcounter{propcounter}
   \begin{enumerate}[label = {{{\normalfont\textbf{\Alph{propcounter}\arabic{enumi}}}}}]
        \item $G'$ is a robust $(\nu, \tau)$-out-expander and a robust $(\nu, \tau)$-in-expander; \label{findrobustexpanderdense:1}
        \item every vertex $v \in V(G')$ satisfies $\deg_{G'}^+(v) \geq \deg_G^+(v) - \delta n$ and $\deg_{G'}^-(v) \geq \deg_G^-(v) - \delta n$. \label{findrobustexpanderdense:2}
        \item $\frac{1}{|V(G')|}\sum_{v \in V(G')} w(v) \leq 2\frac{1}{n}\sum_{v \in V(G)} w(v)$. \label{findrobustexpanderdense:3}
    \end{enumerate}
\end{lemma}
\textbf{Remark.} The preceding lemma may be applied to undirected graphs with minimum degree $\Omega(n)$ as well, by way of their associated symmetric digraphs, and the condition that $|\deg^+(v) - \deg^-(v)| \leq \gamma n$ is satisfied automatically for any~$\gamma$. For this special case, we say that an undirected graph $G$ is  a \emph{robust $(\nu,\tau)$-expander} if the associated symmetric digraph is a robust $(\nu,\tau)$-out-expander.

\begin{proof}
%As roughly sketched above, our plan is to search for a subgraph $G'$ in which all cuts are moderately dense -- this then implies that $G'$ is a robust expander by a simple counting argument that we will give below. Assuming that there is a sparse cut, we focus our attention on one of the two parts of the cut, and we make this choice based on whichever part maintains the property \ref{findrobustexpanderdense:3} better. Throughout the process, the average degree of the subgraph we are building does not change dramatically, as the only cuts we delete are fairly sparse. At the end of the process, cleaning up the few vertices of the subgraph produced (which by definition has no further sparse cuts) that have degree deviating significantly from the average value, we find our desired subgraph $G'$. The formal details follow.

    We may assume without loss of generality that $\tau \leq 1/2$ and that $\alpha \leq 1$. Let $t = \lfloor\log_{1-\tau/2}(\alpha/2) \rfloor,$ and let
    \[\gamma, \nu \ll \rho_0 \ll \rho_1 \ll \cdots \ll \rho_t \ll \delta, \tau, \alpha.\]
    Let $G_0 := G$ and $n_0 := n$.  
    
    For each $0 \leq i \leq t$, in turn, locate a vertex set $U_i \subseteq V(G_i)$ satisfying $\tau n_i/2 \leq |U_i| \leq (1-\tau/2)n_i$ and 
    \[\partial^+_{G_i}(U_i) + \partial^-_{G_i}(U_i) \leq \rho_i n^2,\]
    if such a set exists.
    In this case, define $G_{i+1} := G_i[U_i]$ and $n_{i+1} := |V(G_{i+1})|$.  By replacing $U_i$ by its complement in $G_i$ if necessary, we can assume that
    \begin{equation}\label{eq:average weight inheritance}
    \frac{1}{|U_i|}\sum_{v \in U_i} w(v) \leq \frac{1}{|V(G_i)|}\sum_{v \in V(G_i)}w(v).
    \end{equation}
    If such a set does not exist, then terminate this process, and define $s := i$. If, on the other hand, we manage to find a suitable set $U_i$ for every $0 \leq i \leq t$, then define $s := t+1$. 

    First, we observe that we cannot have $n_s \leq \alpha n/2$. Suppose to the contrary, and consider the minimum index $i$ for which $n_i \leq \alpha n/2$. Note that we also have $n_i \geq \tau n_{i-1}/2 > \tau \alpha n/4$ by the minimality of $i$, which means
    \[\partial_G^+(V(G_{i})) \leq \sum_{j=0}^{i-1} \partial^+_{G_j}(U_j) \leq \sum_{j=0}^{i-1} \rho_j n^2 \leq 2 \rho_{i-1} n^2 \leq \frac{\tau \alpha^2 n^2}{8} < \frac{\alpha n n_{i}}{2}\]
    since $\rho_0 \ll \cdots \ll \rho_{i-1} \ll \tau, \alpha$. But then
    \[\alpha n n_i \leq \sum_{v \in V(G_i)} \deg_G^+(v) = e(V(G_i)) + \partial^+_G(V(G_i)) < n_i^2 + \frac{\alpha n n_i}{2} \leq \alpha n n_i,\]
    which is clearly false. We conclude that $n_s > \alpha n/2$. In particular, this implies that we cannot have $s=t+1$, as $\alpha n/2 < n_s \leq (1-\tau/2)^s n$,
    so $s < \log_{1-\tau/2}(\alpha/2) < t+1$.

    Now by our choice of $s$, we must have been unable to locate a suitable set $U_i$ when we reached $i=s$. Therefore, we have for every $U \subseteq V(G_{s})$ with $\tau n_s/2 \leq |U| \leq (1-\tau/2) n_s$ that
    \[\partial_{G_s}^+(U) + \partial^-_{G_s}(U) \geq \rho_{s} n^2.\]
    We also have that
    \begin{align*}
        |\partial^+_{G_s}(U) - \partial^-_{G_s}(U)| &\leq
        |\partial_G^+(U) - \partial_G^-(U)| + \partial_G^+(V(G_s)) + \partial_G^-(V(G_s)) \\
        &\leq \left|\sum_{v \in U} \deg_G^+(v) - \deg_G^-(v)\right| + \sum_{j=0}^{s-1}\rho_j n^2 \\
        &\leq \gamma n^2 + 2\rho_{s-1}n^2 \leq \frac{\rho_s n^2}{2},
    \end{align*}
    since $\gamma \ll \rho_0 \ll \cdots \ll \rho_{s-1} \ll \rho_{s}$.
    Thus $\partial_{G_s}^+(U) \geq \rho_{s}n^2/4$, and $\partial_{G_s}^-(U) \geq \rho_s n^2/4$.

    We now obtain $G'$ from $G_s$ by deleting the set $Y$ of vertices $v \in V(G_{s})$ with 
    \[\deg^+_G(v,V(G) \setminus V(G_{s})) + \deg^-_G(v,V(G) \setminus V(G_{s})) \geq \sqrt{\rho_{s-1}} n.\] We claim that $G'$ satisfies \ref{findrobustexpanderdense:1}, \ref{findrobustexpanderdense:2}, and \ref{findrobustexpanderdense:3} with $\nu := \rho_s/16$. Note that
    \[|Y|\sqrt{\rho_{s-1}}n \leq \partial^+_G(V(G_s)) + \partial^-_G(V(G_s)) \leq \sum_{j=0}^{s-1}\rho_j n^2 \leq 2\rho_{s-1}n^2,\]
    so $|Y| \leq 2\sqrt{\rho_{s-1}}n \leq \alpha n/4 \leq n_s/2$ since $\rho_{s-1} \ll \alpha$. Thus every vertex $v \in V(G')$ has in-degree at least $\deg_G^-(v) - 3\sqrt{\rho_{s-1}}n \geq \deg_G^-(v) - \delta n$ and out-degree at least $\deg_G^+(v) - 3\sqrt{\rho_{s-1}}n \geq \deg_G^+(v) - \delta n$ since $\rho_{s-1} \ll \delta$. This verifies \ref{findrobustexpanderdense:2}. We also have 
    \[|V(G')| = n_s - |Y| \geq n_s - 2\sqrt{\rho_{s-1}}n \geq n_s/2\]
    since $n_s > \alpha n/2$ and $\rho_{s-1} \ll \alpha$.
    Consequently, if $U \subseteq V(G')$ with $\tau|V(G')| \leq |U| \leq (1-\tau)|V(G')|$, then $\tau n_s/2 \leq |U| \leq (1-\tau/2) n_s$, so 
    \[\partial^+_{G'}(U) \geq \partial^+_{G_s}(U) - |Y|n \geq \frac{\rho_s n^2}{4} - \sqrt{\rho_{s-1}}n^2 \geq \frac{\rho_s n^2}{8}\]
    since $\rho_{s-1} \ll \rho_s$. Then, by the definition of $RN^+_{\nu, G'}(U)$ and since $\nu \ll \rho_s$,
    \[2\nu |V(G')|^2 \leq \frac{\rho_s n^2}{8} \leq \partial^+_{G'}(U) \leq |RN^+_{\nu, G'}(U) \setminus U|\cdot |V(G')| + \nu|V(G')|^2,\]
    which gives $|RN^+_{\nu, G'}(U) \setminus U| \geq \nu |V(G')|$. That is, $G'$ is a robust $(\nu, \tau)$-out-expander, as claimed. A similar argument shows that $G'$ is a robust $(\nu, \tau)$-in-expander as well, so \ref{findrobustexpanderdense:1} is satisfied. To verify \ref{findrobustexpanderdense:3}, we note that \eqref{eq:average weight inheritance} implies by induction that 
    \[\frac{1}{n_s}\sum_{v \in V(G_s)}w(v) \leq \frac{1}{n}\sum_{v \in V(G)}w(v).\]
    Now because $|V(G')| \geq n_s/2$, we have
    \[\frac{1}{|V(G')|}\sum_{v \in V(G')}w(v) \leq \frac{2}{n_s}\sum_{v \in V(G_s)}w(v) \leq \frac{2}{n}\sum_{v \in V(G)}w(v).\]
    This verifies \ref{findrobustexpanderdense:3}, and the proof is complete.
\end{proof}

%\textbf{Remark.} We note that \ref{findrobustexpanderdense:3} has not, to the best of our knowledge, been observed in the literature and holds potential for further applications. One should think of it as saying that if you pick \emph{any} weight function on the vertices of the host graph, then when finding our expander, we can ensure that the average weight on the expander is still essentially upper bounded by the average weight of the host. 
%It offers a way to pass to an expander but maintain (upper bound) control of an arbitrary (vertex-based) statistic, a feature which is very beneficial in a number of different instances where one wants to apply a pass-to-expander lemma.

\subsection{Robust connectivity of robust expanders}\label{sec:connectivity}
The purpose of this section is to formalise the robust connectivity properties of properly edge-coloured dense digraphs whose underlying uncoloured digraph forms a robust expander. Our goal is to conclude that in such a digraph, there is a rainbow path connecting any two vertices $u$ and $v$ whose internal vertices and colours are restricted to some randomly sampled sets, and further restricted to avoid a potentially adversarially chosen set of deterministic vertices and colours. Avoiding given vertices and colours is quite useful, as we will invoke this property for several pairs of vertices $(u_i,v_i)$ consecutively to connect a rainbow path forest, so it is critical not to reuse vertices and colours saturated on a previous connection.

\begin{lemma}[Connecting lemma]\label{lem:connecting lemma}
    Let $n$ be a positive integer, and let $\nu, \tau, \alpha, p \leq 1$ be positive constants satisfying $\nu + \tau \leq \alpha$ and $p^3 \nu^2 n \geq 144\log n$. Define $\beta := p^3 \nu/100$. 
   
    Let $G$ be a properly edge-coloured directed graph on $n$ vertices. Suppose that $G$ is a robust $(\nu, \tau)$-out-expander,
    with $\delta^\pm(G) \geq \alpha n$. Let $V_0, C_0$ be independent $p$-random subsets of $V(G), C(G)$, respectively. Then with probability at least $1 - 5/n$, the following all hold:
   \stepcounter{propcounter}
   \begin{enumerate}[label = {{{\normalfont\textbf{\Alph{propcounter}\arabic{enumi}}}}}]
        \item $|V_0| \leq 2pn$; \label{connectprop:1}
        \item 
        every vertex $v \in V(G)$ satisfies 
        \[\deg^+(v, V(G) \setminus V_0; C(G) \setminus C_0) \geq (1-3p)\deg^+(v);\]
        \[\deg^-(v, V(G) \setminus V_0; C(G) \setminus C_0) \geq (1-3p)\deg^-(v);\]\label{connectprop:2}
        \item for any two distinct vertices $u,v \in V(G)$, and for any vertex subset $V_1 \subseteq V_0$ and colour subset $C_1 \subseteq C_0$, each of size at most $\beta n$, there exists a rainbow directed path of length at most $\nu^{-1} + 1$ from $u$ to $v$ in $G$ whose internal vertices are in $V_0 \setminus V_1$ and whose colours are in $C_0 \setminus C_1$.\label{connectprop:3}
    \end{enumerate}
\end{lemma}

Our intended application  of Lemma~\ref{lem:connecting lemma} is to set aside the randomly sampled sets $V_0$ and $C_0$ as a ``reservoir'', and the remaining coloured directed graph $G'$ will inherit the essential properties of $G$. Using the results of the previous section, we can find a large rainbow path forest in $G'$, and then we can use the reservoir to connect its components into a single rainbow path in $G$ (see \Cref{thm:main} and \Cref{lem:long path in robust expander}).

\par We remark that similar statements about the connectivity of robust expanders  have appeared frequently in the literature. Indeed, such statements without random sampling in the dense setting can be found in~\cite[Proposition 18]{gruslys2021cycle} and \cite{kuhn2015robust} for undirected uncoloured graphs, and with random sampling in sparser settings can be found in~\cite{erdos-gallai} for uncoloured graphs and in~\cite{rainbow-tomon} for coloured graphs. The presence of colours and the setting of directed graphs in our case makes certain aspects of the following argument notationally more technical but here is the gist of the argument. In an uncoloured robust out-expander with linear in and out degree, if one grows a BFS tree rooted at any fixed vertex, in a bounded number of steps we reach all the vertices, which  allows to connect any two vertices $u$ and $v$ with a short path. In the coloured setting, we have to keep track of the colours of the edges used on each subpath we obtain in the BFS, but using the proper edge-colouring of the digraph this is not an issue.

%An alternative way to do this for undirected graphs, is growing BFS trees from any two fixed vertices and stop when these trees intersect. However, for directed graphs, our method is more versatile, in particular, it only uses the out-expansion of digraphs.

\begin{proof}[ of Lemma~\ref{lem:connecting lemma}]
    The properties \ref{connectprop:1} and \ref{connectprop:2} occur with high probability by Chernoff's inequality (\Cref{chernoff}). Indeed, the probability that \ref{connectprop:1} fails is at most
    \[\exp\left(-\frac{pn}{3}\right) \leq n^{-48}\]
    since $pn \geq p^3\nu^2 n \geq 144 \log n$.
    As for \ref{connectprop:2}, we have for every $v \in V(G)$ that each of $\deg^+(v) - \deg^+(v, V(G) \setminus V_0; C(G) \setminus C_0)$ and $\deg^-(v) - \deg^-(v,V_0; C(G) \setminus C_0)$ is a binomial random variable with expectation $(2p-p^2)\deg^+(v)$ and $(2p-p^2)\deg^-(v)$, respectively.
    Therefore, because $\delta^\pm(G) \geq \alpha n$, the probability that \ref{connectprop:2} fails is at most 
    \[2n \exp\left(-\frac{(p + p^2)^2\delta^\pm(G)}{2(2p-p^2) + (p + p^2)}\right) \leq 2n \exp\left(-\frac{p \alpha n}{5}\right) \leq 2n^{-139/5}\]
    since $p\alpha n \geq p^3 \nu^2 n \geq 144\log n$.

    In order to ensure \ref{connectprop:3}, we first need to introduce some notation.
    
    \par For a vertex $u \in V(G)$, define $N_1(u) := N_G^+(u)$, and for each $i \geq 1$, define $N_{i+1}(u) := N_i(u) \cup RN_{\nu, G}^+(N_i(u))$. Then $|N_1(u)| \geq \alpha n \geq \tau n$, so because $G$ is a robust $(\nu, \tau)$-out-expander,
    we have $|N_{i+1}(u)| \geq |N_i(u)| + \nu n$ as long as $|N_i(u)| \leq (1-\tau)n$. Thus there exists $i_0 \leq \nu^{-1}$ such that $|N_{i_0}(u)| \geq (1-\tau)n$. As any vertex in $V(G)$ has at least $\alpha n - \tau n \geq \nu n$ in-neighbours in $N_{i_0}(u)$, we have that $RN_{\nu, G}^+(N_{i_0}(u)) = V(G)$.

    Given $u,w \in V(G)$ with $w \in RN_{\nu/2, G}^+(N_1(u))$, let $Y_{u,w,1}$ be the number of in-neighbours $x$ of $w$ in $N_1(u)$ such that $x \in V_0$, and $C(u,x)$ and $C(x,w)$ are both in $C_0$. Similarly, if $w \in RN_{\nu/2, G}^+(N_i(u) \setminus N_{1}(u))$ for some $2 \leq i \leq i_0$, let $Y_{u,w,i}$ be the number of in-neighbours $x$ of $w$ in $N_i(u) \setminus N_1(u)$ such that $x \in V_0$, and $C(x,w) \in C_0$. We will ensure that 
    \begin{equation}\label{eq:many good neighbours 1}
        Y_{u,w,1} \geq 3(\beta n + \nu^{-1}) + 1
    \end{equation}
    whenever $w \in RN_{\nu/2, G}^+(N_1(u))$, and that
    \begin{equation}\label{eq:many good neighbours i}
        Y_{u,w,i} \geq 2(\beta n + \nu^{-1}) + 1
    \end{equation}
    whenever $w \in RN_{\nu/2, G}^+(N_i(u) \setminus N_{1}(u))$.

    For $2 \leq i \leq i_0$ and $w \in RN_{\nu/2, G}^+(N_i(u) \setminus N_{1}(u))$, $Y_{u,w,i}$ is a binomial random variable with expectation at least 
    \[\frac{p^2 \nu n}{2} \geq \frac{p^3 \nu n}{4} + \frac{p^3 \nu n}{4} \geq 25\beta n + 36\nu^{-1} \geq 2(2(\beta n + \nu^{-1}) + 1),\]
    since $p^3 \nu^2 n \geq 144 \log n \geq 144.$
    Thus, the probability that some $Y_{u,w,i} < 2(\beta n + \nu^{-1}) + 1$ for $i \geq 2$ is at most 
    \[\nu^{-1} n^2 \exp\left(-\frac{p^2 \nu n}{16}\right) \leq n^{-6}\]
    by Chernoff's inequality (\Cref{chernoff}) and a union bound. Here, we used that $\nu n \geq p^3 \nu^2 n \geq 144 \log n \geq 1$, so $\nu^{-1} \leq n$.

    If $w \in RN_{\nu/2, G}^+(N_1(u))$, then we can greedily obtain a set $W_{u,w} \subseteq N_G^-(w) \cap N_1(u)$ of size at least $\nu n/6$ such that all $2|W_{u,w}|$ edges of the form $(u,x)$ or $(x,w)$ for $x \in W_{u,w}$ have distinct colours. 
    Now $Y_{u,w,1}$ is at least the number of vertices $x \in W_{u,w} \cap V_0$ such that $C(u,x)$ and $C(x,w)$ are both in $C_0$, which is a binomial random variable with expectation at least 
    \[\frac{p^3\nu n}{6} = \frac{p^3\nu n}{12} + \frac{p^3\nu n}{12} \geq \frac{25 \beta n}{3} + 12 \nu^{-1} \geq 2(3(\beta n + \nu^{-1}) + 1).\] Thus the probability that some $Y_{u,w,1} < 3(\beta n + \nu^{-1}) + 1$ is at most
    \[n^2 \exp\left(-\frac{p^3\nu n}{48}\right) \leq n^{-1}.\]

    We now fix an outcome $V_0$ and $C_0$ for which  \ref{connectprop:1}, \ref{connectprop:2},  (\ref{eq:many good neighbours 1}), and (\ref{eq:many good neighbours i}) all hold (for any $u,w,i$ in the case of (\ref{eq:many good neighbours 1}) and (\ref{eq:many good neighbours i})). We will now show that whenever these conditions hold, \ref{connectprop:3} is also satisfied. Since we have shown that these conditions hold with probability at least $1-n^{-48}-2n^{-139/5}-n^{-6}-n^{-1} \ge 1-5/n$, this will complete the proof.

     Fix two distinct vertices $u,v \in V(G)$, a vertex subset $V_1 \subseteq V_0$, and a colour subset $C_1 \subseteq C_0$, each of size at most $\beta n$. Our goal is to find a rainbow directed path $P$ from $u$ to $v$ of length at most $\nu^{-1} + 1$, with $V(P) \setminus \{u,v\} \subseteq V_0 \setminus V_1$ and $C(P) \subseteq C_0 \setminus C_1$.

    \begin{claim*}
        There exist a vertex $w \in RN^+_{\nu/2, G}(N_1(u))$ and a rainbow directed path $Q$ from $w$ to $v$ of length at most $\nu^{-1} - 1$, with $V(Q) \setminus \{v\} \subseteq V_0 \setminus V_1$ and $C(Q) \subseteq C_0 \setminus C_1$.
    \end{claim*}
    \begin{cla_proof}
    Recall that $v \in RN_{\nu, G}^+(N_{i_0}(u))=V(G)$. Define $v_0 := v$. Let $k \geq 0$ be maximal so that there exists some rainbow directed path $P_k := (v_k,v_{k-1},\ldots ,v_0)$ such that for indices $1 \leq i_k < i_{k-1} < \cdots < i_0$, we have 
    \begin{itemize}
        \item $v_j \in RN_{\nu,G}^+(N_{i_j}(u))$ for $0 \leq j \leq k$;
        \item $V(P_k) \setminus \{v_0\} \subseteq V_0 \setminus V_1$;
        \item $C(P_k) \subseteq C_0 \setminus C_1$.
    \end{itemize}
    The trivial path at $v_0$ satisfies all of these conditions with $k=0$, so $P_k$ is well-defined.
    We will show that the claim is satisfied with $w := v_k$ and $Q := P_k$. First note that $i_1 > \cdots > i_k$ are $k$ distinct integers in $[1,i_0-1]$, so the length of $Q$ is $k \leq i_0 - 1 \leq \nu^{-1} - 1$.
    It remains to show that $w$ has at least $\nu n/2$ in-neighbours in $N_1(u)$.
    
    Suppose not. Then since $w = v_k \in RN^+_{\nu, G}(N_{i_k}(u))$, we must have $i_k \geq 2$ and $v_k \in RN_{\nu/2,G}^+(N_{i_k}(u) \setminus N_1(u))$. By (\ref{eq:many good neighbours i}), $v_k$ has at least $2(\beta n + \nu^{-1}) + 1$ in-neighbours $x \in N_{i_k}(u) \setminus N_1(u)$ such that $x \in V_0$ and $C(x,w) \in C_0$. Of these, at most $\beta n$ are in $V_1$, at most $\beta n$ have $C(x,w) \in C_1$, at most $k+1 \leq \nu^{-1}$ are already in $V(P_k)$, and at most $k \leq \nu^{-1}$ have $C(x,w) \in C(P_k)$. This leaves at least one choice of an in-neighbour $v_{k+1} \in N_{i_k}(u) \setminus N_1(u)$ for which $P_{k+1} := (v_{k+1}, \ldots, v_0)$ is a rainbow directed path with $V(P_{k+1}) \setminus \{v_0\} \subseteq V_0 \setminus V_1$ and $C(P_{k+1}) \subseteq C_0 \setminus C_1$. By definition of $N_{i_k}(u) \setminus N_1(u)$, we have that $v_{k+1} \in RN^+_{\nu, G}(N_{i_{k+1}}(u))$ for some $1 \leq i_{k+1} < i_k$, contradicting the maximality of $k$.
    
    We conclude that $w \in RN_{\nu/2,G}^+(N_1(u))$. This proves the claim. 
    \end{cla_proof}

    Let $w$ and $Q$ be as in the claim. If $u \in V(Q)$, then we can take $P$ to be the subpath of $Q$ from $u$ to $v$. If not, then by (\ref{eq:many good neighbours 1}), $w$ has at least $3(\beta n + \nu^{-1}) + 1$ in-neighbours $x \in N_1(u)$ such that $x \in V_0$, and $C(u,x), C(x,w) \in C_0$. Of these, at most $\beta n$ are in $V_1$, at most $2\beta n$ have $C(u,x) \in C_1$ or $C(x,w) \in C_1$, at most $|V(Q)| \leq \nu^{-1}$ are already in $V(Q)$, and at most $2|C(Q)| \leq 2\nu^{-1}$ have $C(u,x) \in C(Q)$ or $C(x,w) \in C(Q)$. This leaves at least one choice for $x$ in which $P := (u,x,Q)$ is a rainbow directed path from $u$ to $v$ with $V(P) \setminus \{u,v\} \subseteq V_0 \setminus V_1$ and $C(P) \subseteq C_0 \setminus C_1$. Because the length of $P$ is at most $\nu^{-1} + 1$, the proof is complete.
\end{proof}

\section{Long rainbow paths in dense digraphs}\label{sec:dense}
In this section we prove Theorem~\ref{thm:regulardigraph-intro} and \Cref{thm:summary}(c). The rough idea of both proofs is as follows. Dense regular digraphs (and more generally Eulerian digraphs) without good expansion properties have sparse cuts, so we can always zoom into an expanding subgraph with essentially the same degrees as the original digraph (as in \Cref{lem:find robust expander}). On the other hand, dense digraphs \textit{with} good expansion properties (including all sufficiently dense Cayley graphs on $\mathbb Z_p$ with $p$ prime, as we will see) are well-connected, so we can build a rainbow path by first finding a large rainbow path forest with few components via \Cref{lem:rainbow path forest}, and then connecting the components together via short rainbow paths via \Cref{lem:connecting lemma}. 

We begin with the case where our dense digraph already has robust expansion properties.

\begin{lemma}\label{lem:long path in robust expander}
    For any $\eps > 0$, there exists a positive integer $n_0$ such that the following holds for all $n \geq n_0$. Let $\nu, \tau, \alpha > 0$ satisfying $\nu + \tau \leq \alpha \leq 1$ and $\nu \geq n^{-1/2}\log n$. Let $G$ be a properly edge-coloured directed graph on $n$ vertices. Suppose that $G$ is a robust $(\nu, \tau)$-out-expander with $\delta^\pm(G) \geq \alpha n$. Then $G$ contains a rainbow directed path of length at least $\left(1 - \eps\right)\delta^{\pm}(G)$. 
\end{lemma}

\begin{proof}
We may assume $\eps \leq 1$. Let $d := \delta^{\pm}(G)$. Set $p := \eps/6$ and $\beta := p^3 \nu/100$. Since $\nu \geq n^{-1/2}\log n$ and $1/n \ll \eps$, we have
\[p^3 \nu^2 n \geq \frac{\eps^3 \log^2 n}{216} \geq 144\log n,\]
so by \Cref{lem:connecting lemma}, there exist sets $V_0 \subseteq V(G)$ and $C_0 \subseteq C(G)$ satisfying \ref{connectprop:1}--\ref{connectprop:3}. Let $G' := G[V(G) \setminus V_0; C(G) \setminus C_0]$. By \ref{connectprop:2}, the minimum in-degree of $G'$ is at least \[(1-3p)d \geq (1-\eps/2)\alpha n \geq \frac{\nu n}{2} \geq \frac{\sqrt{n} \log n}{2} \geq 9(\eps/2)^{-3}\]
since $\alpha \geq \nu \geq n^{-1/2}\log n$, and $1/n \ll \eps \leq 1$. Therefore, by \Cref{lem:rainbow path forest} applied to $G'$ with $\eps/2$ in place of $\eps$, $G'$ has a rainbow path forest $\mathcal P$ with at most $36\eps^{-2}$ components and at least $(1-\eps/2)(1-3p)d \geq (1-\eps) d$ edges.

Let $P_1, \ldots, P_m$ be the components of $\mathcal P$, where $m \leq 36\eps^{-2}$. For each $1 \leq i \leq m$, let $u_i$ denote the vertex of $P_i$ with in-degree $0$, and let $w_i$ denote the vertex of $P_i$ with out-degree $0$. Now for each $1 \leq i \leq m-1$, in turn, find a path $Q_i$ from $w_i$ to $u_{i+1}$ of length at most $\nu^{-1} + 1$ whose internal vertices are in $V_0 \setminus \bigcup_{j < i} V(Q_j)$ and whose colours are in $C_0 \setminus \bigcup_{j < i} C(Q_j)$. Such paths exist by \ref{connectprop:3}, as for each $i$, the sizes of $\bigcup_{j < i} V(Q_j)$ and $\bigcup_{j < i} C(Q_j)$ never exceed 
\[m\left(\nu^{-1} + 2\right) \leq 36\eps^{-2}\left(\frac{\sqrt n}{\log n} + 2\right) \leq \frac{\eps^3 \sqrt n \log n}{21600}\leq \frac{p^3 \nu n}{100} = \beta n\]
since $\nu \geq n^{-1/2} \log n$ and $1/n \ll \eps$. Now $\bigcup_{i=1}^m P_i \cup \bigcup_{i=1}^{m-1} Q_i$ is a rainbow path in $G$ of length at least $(1-\eps)d$, as required. \end{proof}

\subsection{Dense digraphs}

We now prove Theorem~\ref{thm:regulardigraph-intro} in the following stronger form. Here we relax the condition that the digraph $G$ is $d$-regular with $d = \Omega(n)$ to the condition that $G$ is almost-Eulerian (i.e., each vertex has in-degree within $o(n)$ of its out-degree) with $\delta^\pm(G) = \Omega(n)$.

\begin{theorem}\label{theorem:rainbowdirectedpathindensecase}
    For any $\eps, \alpha >0$, there exist $n_0 \in \mathbb N$ and $\gamma > 0$ such that the following holds for all $n \geq n_0$. Let $G$ be a properly edge-coloured directed graph on $n$ vertices with $d := \delta^\pm(G) \geq \alpha n$. Suppose that every vertex $v \in V(G)$ satisfies $|\deg^+(v) - \deg^-(v)| \leq \gamma n$. Then $G$ has a rainbow directed path of length at least $(1-\eps)d$.
\end{theorem}

\begin{proof}
 We may assume without loss of generality that $\eps \leq 1$. Let $\delta := \eps \alpha/2$, and let $\tau, \nu, \gamma > 0$ be such that $1/n \ll \gamma, \nu \ll \tau \ll \eps, \alpha$. Let $G'$ be as in \Cref{lem:find robust expander}, satisfying \ref{findrobustexpanderdense:1} and \ref{findrobustexpanderdense:2} (we don't need \ref{findrobustexpanderdense:3} for this application).
    By \ref{findrobustexpanderdense:2}, we have $\delta^\pm(G') \geq d - \eps \alpha n/2 \geq (1-\eps/2)d \geq \alpha n/2$.
    We also have $\nu \geq |V(G')|^{-1/2}\log|V(G')|$ and $\nu + \tau \leq (1-\eps/2)\alpha$ since $|V(G')| \geq \delta^{\pm}(G') \geq \alpha n/2$ and $1/n \ll \nu, \tau \ll \alpha, \eps \leq 1$. Now by \Cref{lem:long path in robust expander} with $\eps/2$, $|V(G')|$, $(1 - \eps/2)\alpha$, $G'$, in place of $\eps$, $n$, $\alpha$, $G$, respectively, $G'$ contains a rainbow directed path of length at least
    \[\left(1 - \eps/2\right)\delta^\pm(G') \geq (1 - \eps) d.\]
    This completes the proof.\end{proof}

\textbf{Remark.} Essentially the same proof also guarantees the existence of a directed cycle of length $d-o(d)$ in both \Cref{lem:long path in robust expander} and \Cref{theorem:rainbowdirectedpathindensecase}. Moreover, our proof of \Cref{theorem:rainbowdirectedpathindensecase} can be optimised so that the conclusion holds whenever $\delta^\pm(G) \geq 100n \log \log n/\log n$, but to keep our computations neat, we prove this slightly weaker version assuming $\delta^\pm(G) = \Omega(n)$.

\subsection{Graham's original problem for \texorpdfstring{$\mathbb{Z}_p$}{Zp}}
\label{sec:Z_p}

In this section, we prove \Cref{thm:summary}(c) as an immediate consequence of \Cref{lem:long path in robust expander}. We will use a classical theorem of Pollard from 1970 \cite{Pollard} generalising the classical Cauchy-Davenport inequality to show that any dense Cayley graph on $\mathbb{Z}_n$ for prime $n$ is a robust out-expander\footnote{Throughout this section, we denote the (prime) order of the group by $n$ instead of $p$ in order to keep consistency with the rest of the paper.}. To state it, we first require some notation. Given two subsets $A,B$ of an Abelian group $\Gamma$ and an element $x\in \Gamma$, we denote by $r_{A,B}(x)$ the number of $(a,b)\in A\times B$ such that $x=a+b$. Let us define the \emph{$t$-representable sum} of $A$ and $B$ as $$A+_t B:=\{x \in \Gamma \mid r_{A,B}(x) \ge t\}.$$ 
We are now ready to state Pollard's theorem.

\begin{theorem}\label{thm:pollard}
    Let $n$ be a prime and $A,B$ be nonempty subsets of $\mathbb{Z}_n$. Then for any $t \le \min\{|A|,|B|\}$ we have
    $$\sum_{i=1}^{t} |A+_iB| \ge t \cdot \min \{n,|A|+|B|-t\}.$$ 
\end{theorem}

We now prove \Cref{thm:summary}(c). We restate the theorem in terms of rainbow paths in Cayley graphs.

\begin{theorem}\label{thm:Zp}
    For any $\eps > 0$, there exists $n_0$ such that the following holds for all primes $n \geq n_0$. Let $S \subseteq \mathbb{Z}_n \setminus \{0\}$ of size $|S| \ge 4n^{3/4}\sqrt{\log n}$. Then $\mathrm{Cay}(\mathbb{Z}_n,S)$ contains a rainbow path of length $(1-\eps)|S|$.
\end{theorem}

\begin{proof}
    Let $\eps > 0$, and assume that $n$ is a sufficiently large prime. Let $S \subseteq \mathbb Z_n \setminus \{0\}$ of size $d := |S| \ge 4n^{3/4}\sqrt{\log n}$, and let $G := \mathrm{Cay}(\mathbb Z_n, S)$. Also define $\nu := d^2/(8n^2)$, and let $\tau := d/(2n)$. We will first show that $G$ is a robust $(\nu, \tau)$-out-expander. Let $U \subseteq V(G)$ with $\tau n \leq |U| \leq (1-\tau)n$, and let $t := \lceil d/2 \rceil$. Note that $|S|, |U| \geq t$, so we can apply \Cref{thm:pollard} with $U$, $S$ in place of $A$, $B$ to conclude that
    \[t(|U| + |S| - t) = t \cdot \min \left\{n, |U| + |S| - t\right\} \leq \sum_{i=1}^t |U +_i S| = \sum_{x \in \mathbb Z_n} \min\left\{t, r_{U,S}(x)\right\} = \sum_{x \in V(G)} \min \left\{t, \deg^-_G(x,U)\right\},\]
    where the first equality used that $|U| \leq (1-\tau)n$, so $|U| + |S| - t \leq |U| + d/2 \leq n$. Defining $N := RN^+_{\nu,G}(U) \setminus U$, we can estimate the rightmost sum by
    \[\sum_{x \in V(G)} \min \left\{t, \deg^-_G(x,U)\right\} \leq t|U| + t\left|N\right| + \nu n(n - |U| - |N|) \leq t|U| + t|N| + \nu n^2,\]
    which gives
    \[|N| \geq \frac{t(|S| - t) - \nu n^2}{t} \geq \left\lfloor d/2 \right\rfloor - d/4 \geq d^2/(8n) = \nu n\]
    since $4 \leq d \leq n$. Hence $G$ is a robust $(\nu, \tau)$-out-expander.

    Now set $\alpha := d/n$ so $\delta^{\pm}(G) = d = \alpha n$, and note that $\nu + \tau \leq \alpha$ and $\nu = d^2/(8n^2) > n^{-1/2}\log n$. By \Cref{lem:long path in robust expander}, $G$ contains a rainbow directed path of length at least $(1-\eps)d$.
\end{proof}

We note that there has been plenty of work on generalising Pollard's Theorem to other Abelian groups. One can use a generalization due to Green and Ruzsa \cite{pollard-generalized} (see also \cite{Grynkiewicz2013}) to extend \Cref{thm:Zp} to any Abelian group $\Gamma$, but with the density requirement depending on the size of the largest proper subgroup of $\Gamma$.

\section{Long rainbow paths in regular graphs}\label{sec:mop}
In this section, we prove \Cref{thm:schrijver-asymptotic-intro}, which says that we can find a long rainbow path in any regular properly edge-coloured (undirected) graph. The proof is split into two parts. In the first part (\Cref{sec:expanding case}), we show that if the graph expands in a particular way, then we can iteratively grow our desired long rainbow path from a single vertex. In the second part (\Cref{sec:non-expanding case}), we show that when this expansion property fails, then we can reduce the problem to a dense subgraph and apply the tools of Sections~\ref{sec:pathforests} and~\ref{sec:expansion}. 

The type of expansion that we will look for will be quantified in terms of an auxiliary object called a \emph{mop}, which we now define.

\begin{defn}
    Let $G$ be a properly edge-coloured graph. For non-negative integers $\ell, s, t$,  an $(\ell,s, t)$-\emph{mop} in $G$ is a quadruple $(P,v, U, \{Q_u\}_{u \in U})$ such that $P \subseteq G$ is a rainbow path of length $\ell$ with $v$ as an endpoint, $U \subseteq V(G)$ is a set of $t$ vertices, and for each $u \in U$, $Q_u \subseteq G$ is a path from $v$ to $u$ of length at most $s$ such that $P\cup Q_u$ is a rainbow path. We call $P$ the \emph{handle} of the mop, the vertices in $U$ the \emph{ends}, and the $Q_u$'s the \emph{strands}. See \Cref{fig:mop} for an illustration.
\end{defn}

\begin{figure}[h]
    \centering
    \includegraphics[width=0.7\linewidth]{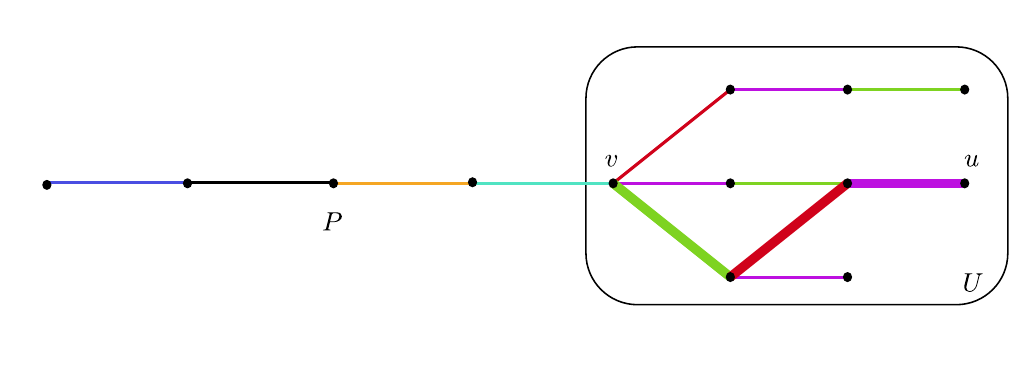}
    \caption{A $(4,3,9)$-mop $(P,v,U,\{Q_u\}_{u \in U})$. Thick edges make a single strand $Q_u$ .}
    \label{fig:mop}
\end{figure}

To give some intuition behind this definition, recall the standard greedy argument for finding a path of length $d$ in a graph with minimum degree $d$. We first take the longest path and argue that if the path was of length $\leq d-1$, then the end-vertex could not send all of its neighbours to the path and hence one can use a neighbour to extend the path, contradicting maximality. The issue with this argument in the rainbow setting is that, in addition to not sending edges to the already built path, we must also not reuse colours appearing on the path (though this is already enough to conclude the existence of a rainbow path of length $\sim d/2$). The way around this is that instead of looking at a single longest path, we look for an object that consists of a long rainbow path (the handle) and many short rainbow continuations (strands) through which we can reach many vertices (ends). The major benefit of this approach is that now we have many potential end-vertices of a long rainbow path instead of just one. This allows one to easily extend one of these paths by a single vertex, but in order to iterate the argument, this is not good enough; we want to preserve the mop structure as well. The following definition captures a weak expansion property we require from the ends in order to be able to perform the iteration.

\begin{defn}
    We say that an $(\ell,s,t)$-mop $(P, v, U, \{Q_u\}_{u \in U})$ in a properly edge-coloured graph $G$ is \emph{$(d, \gamma)$-leaky} if $e(U, V(G) \setminus U; C(G) \setminus C(P)) \geq \gamma (d-\ell)|U|$. We say that $G$ is a \emph{$(d,s_0, \gamma,\alpha)$-leaky mop expander} if every $(\ell,s,t)$-mop in $G$ with $s \leq s_0$ and $t \leq \alpha^{-1}(d-\ell)$ is $(d,\gamma)$-leaky. See \Cref{fig:leaky-mop} for an illustration.
\end{defn}

\begin{figure}[h]
    \centering
    \includegraphics[width=0.7\linewidth]{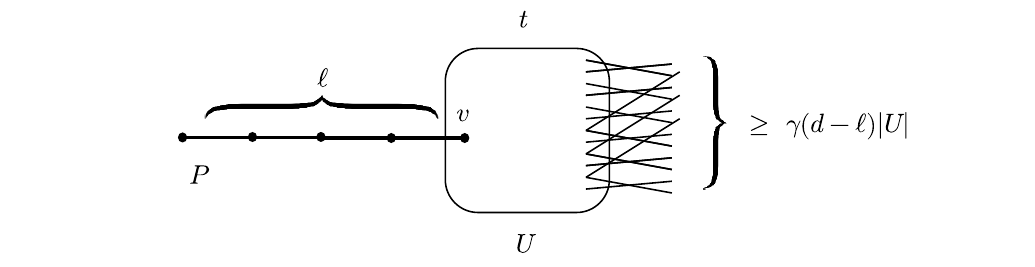}
    \caption{Illustration of a $(d,\gamma)$-leaky mop. We note that the expanding edges may involve vertices, but not colours, from the handle. They may also reuse vertices and colours from the strands.}
    \label{fig:leaky-mop}
\end{figure}

At the highest level of our proof of \Cref{thm:schrijver-asymptotic-intro}, we distinguish two cases: either $G$ is a leaky mop expander with appropriate parameters, or, $G$ is not. In the former case, we use this weak expansion property to build mops with longer and longer handles. In the latter case, we apply the tools from previous sections to the dense subgraph induced by the ends of a mop which is not leaky.

\subsection{Expanding case}\label{sec:expanding case}

Here, we prove that if $G$ is a leaky mop expander, then we can build our desired long path. The basic idea is that a leaky $(\ell,s,t)$-mop with the number $t$ of ends large enough compared to the length $\ell$ of the handle can be made significantly larger (increasing $t$) by increasing the maximum length $s$ of the strands by one. If $t$ is \textit{very} large compared to $\ell$, then we can extend the handle to build a longer mop (increasing $\ell$) by incorporating a strand whose end is sufficiently ``far'' from the handle in an appropriate sense. In a $d$-regular leaky mop expander, this can be iterated until the handle has length $d - o(d)$. 

The following lemma will allow us to increase the number of ends of a mop by allowing the maximum length of its strands to increase by one.

\begin{lemma}\label{lem:more ends}
	Let $d, \ell, s, t, \Delta$ be nonnegative integers, and let $\gamma > 0$. Let $G$ be a properly edge-coloured graph with $\Delta(G) \leq \Delta$. If $G$ contains a $(d,\gamma)$-leaky $(\ell, s, t)$-mop in $G$, then $G$ contains an $(\ell,s+1,t')$-mop with 
	\[t' \geq \left(1 + \frac{\gamma(d - \ell) - 2s}{\Delta}\right)t - (\ell+1).\]
\end{lemma}

Sometimes, it will also be necessary to ensure that the new ends avoid a certain set $B$ of ``bad'' vertices, so we prove the following more general lemma. \Cref{lem:more ends} is simply the special case where $B = \emptyset$ and $\Delta' = 0$. \Cref{fig:boost} illustrates the statement of the lemma.

\begin{figure}[h]
    \centering
    \includegraphics[width=0.7\linewidth]{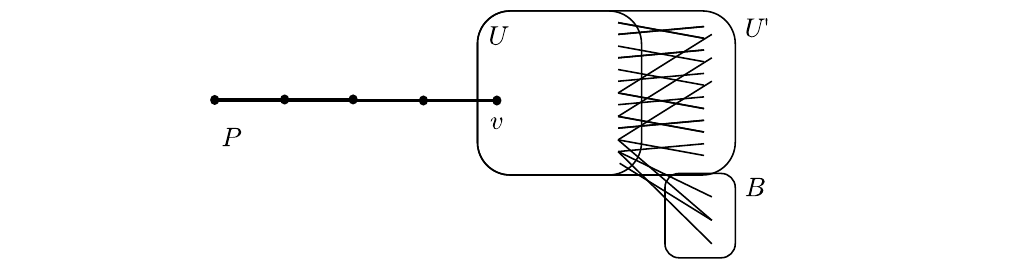}
    \caption{Illustration of the setup in \Cref{lem:more ends bad set}. By taking one step out, using the leaky expansion property, we can reach many new vertices, even while avoiding a given small set of ``bad'' vertices, and taking into account interactions with the handle and current strands.}
    \label{fig:boost}
\end{figure}

\begin{lemma}\label{lem:more ends bad set}
	Let $d, \ell, s, t, \Delta, \Delta'$ be nonnegative integers, and let $\gamma > 0$. Let $G$ be a properly edge-coloured graph with $\Delta(G) \leq \Delta$. Let $(P, v, U, \{Q_u\}_{u \in U})$ be a $(d,\gamma)$-leaky $(\ell, s, t)$-mop in $G$, and let $B \subseteq V(G) \setminus U$ be a set of vertices with the property that for every $u \in U$, we have $\deg(u,B;C(G) \setminus C(P)) \leq \Delta'$. Then $G$ contains an $(\ell,s+1,t')$-mop $(P, v, U', \{\tilde Q_{u'}\}_{u' \in U'})$ with $U' \cap B = \emptyset$ and 
	\[t' \geq \left(1 + \frac{\gamma(d - \ell) - \Delta' - 2s}{\Delta}\right)t - |V(P) \setminus B|.\]
\end{lemma}
\begin{proof}
	Let $C' := C(G) \setminus C(P)$.
    At a high level, our plan is to expand the current set of ends $U$, by colours in $C'$, using our leaky property, and argue that many of the new vertices we reach are suitable ends for our new mop. Let~$W$ be the set of these ``suitable'' vertices: more precisely, the set of vertices $w \in V(G) \setminus (U \cup B)$ with some neighbour $u \in U$ such that $w \notin V(P \cup Q_u)$ and $C(uw) \in C' \setminus C(Q_u)$ (so that for $\tilde Q_{w} := Q_u \cup \{uw\}$, we have that $P \cup \tilde Q_{w}$ is a rainbow path). Then $(P, v, U \cup W, \{Q_u\}_{u \in U} \cup \{\tilde Q_{w}\}_{w \in W}\})$ is a $(0,s+1, |U \cup W|)$-mop with $(U \cup W) \cap B = \emptyset$. It suffices to show that $|W| \geq  \left(\frac{\gamma(d - \ell) - \Delta' - 2s}{\Delta}\right)t - |V(P) \setminus B|$.
	
	We count edges from $U$ to $V(G) \setminus U$ with colours in $C'$. Since $(P, v, U, \{Q_u\}_{u \in U})$ is $(d,\gamma)$-leaky, we have
	\begin{align*}
		\gamma (d-\ell) |U| &\leq e(U, V(G) \setminus U; C') \\
		&\leq e(U, B; C') + e(U, W \cup (V(P) \setminus B); C') + e(U, V(G) \setminus (U \cup W \cup V(P) \cup B); C') \\
		&\leq \Delta'|U| + \Delta |W \cup (V(P) \setminus B)| + e\left(U, V(G) \setminus \left(U \cup W \cup V(P) \cup B\right); C'\right).
	\end{align*}
    We note that in the final expression, the first term is small, and we wish to lower bound the second. We next argue that the third term is also small. It accounts for the edges whose endpoint is blocked by the strand that connects it to the handle. Indeed, by the definition of $W$, for every $uw \in E(G)$ with $u \in U$, $w \in V(G) \setminus (U \cup W \cup V(P) \cup B)$, and $C(uw) \in C'$, we either have that $w \in V(Q_u) \setminus \{u\}$ or $C(uw) \in C(Q_u)$, and so 
	\[e(U, V(G) \setminus (U \cup W \cup V(P) \cup B); C') \leq \sum_{u \in U} \left(|V(Q_u) \setminus \{u\}| + |C(Q_u)|\right) \leq 2s|U|\]
	since $G$ is properly edge-coloured. Thus
	\[\gamma (d-\ell)|U| \leq \Delta'|U| + \Delta\left|W \cup (V(P) \setminus B)\right| + 2s|U|.\]
	Rearranging gives
	\[|W| \geq \left(\frac{\gamma (d-\ell) - \Delta' - 2s}{\Delta}\right)|U| - |V(P) \setminus B| = \left(\frac{\gamma (d-\ell) - \Delta' - 2s}{\Delta}\right)t - |V(P) \setminus B|,\]
    as desired.
\end{proof}

We now prove \Cref{thm:schrijver-asymptotic-intro} in the case where $G$ is a leaky mop expander.

\begin{lemma}\label{lem:mop argument}
    Let $1/d \ll \alpha \ll 1/s_0 \ll \gamma, \eps, 1/K \leq 1$, and let $G$ be a properly edge-coloured graph with $d = \delta(G) \leq \Delta(G) \leq Kd$. If $G$ is a $(d,s_0, \gamma, \alpha)$-leaky mop expander, then $G$ contains a rainbow path of length at least $(1-\eps)d$.
\end{lemma}

\begin{proof}
	Introduce a further parameter $s_1$ so that $1/s_0 \ll 1/s_1 \ll \gamma, \eps, 1/K$. Our strategy is to iteratively apply \Cref{lem:more ends} to build mops with more and more ends, then pick one of the strands carefully, extend our handle with it, and then apply \Cref{lem:more ends bad set} to construct a mop from scratch with this longer handle. Iterating this procedure will eventually produce a mop whose handle has length at least $(1-\eps)d$. We begin by producing mops with many ends whose handle consists of a single vertex.
	\begin{claim*}[Initialise]
		For every $1 \leq s \leq s_0$, there exists a $(0,s,t)$-mop in $G$ with $t \geq \left(1 + \frac{\gamma}{2K}\right)^{s-1}d$.
	\end{claim*}
	\begin{cla_proof}
		We prove the claim by induction on $s$. Taking any vertex $v$ and $d$ of its neighbours gives a $(0,1,d)$-mop, so this proves the base case $s=1$. Now suppose that $2 \leq s \leq s_0$, and assume that we have a $(0,s-1,t')$-mop $M$ with $t' \geq \left(1 + \frac{\gamma}{2K}\right)^{s-2}d$. If $t' \geq \left(1 + \frac{\gamma}{2K}\right)^{s-1}d$, then $M$ is already a $(0,s,t')$-mop satisfying the claim, so we may assume that $t' < \left(1 + \frac{\gamma}{2K}\right)^{s-1}d \leq \alpha^{-1}d$ since $\alpha \ll 1/s_0, \gamma, 1/K$. Since $G$ is a $(d,s_0, \gamma, \alpha)$-leaky mop expander, it follows that $M$ is $(d,\gamma)$-leaky. Now by \Cref{lem:more ends}, there exists a $(0,s,t)$-mop with
		\[t \geq \left(1 + \frac{\gamma d - 2s_0}{Kd}\right)t' - 1 \geq \left(1 + \frac{3\gamma}{4K}\right)t' - 1 \geq \left(1 + \frac{\gamma}{2K}\right)^{s-1}d\]
		since $1/d \ll 1/s_0, \gamma, 1/K$, and $t' \geq \left(1 + \frac{\gamma}{2K}\right)^{s-2}d \geq d$.
	\end{cla_proof}
    
    % We note that in the above claim and in general throughout the proof, we need to maintain an upper bound on the length of the strands (this is the role of the parameter $s_0$), as we maintain very little control of their behaviour. 
    
	We can apply the same argument as we did in the above claim to mops with longer handles, provided we have a sufficient number of ends with which to begin the process.
	\begin{claim*}[Boost]
		Let $0 \leq \ell \leq (1-\eps)d$. If there exists an $(\ell,s_1,t_1)$-mop in $G$ with $t_1 \geq \left(1 + \frac{\gamma \eps}{2K}\right)^{s_1-1}\frac{\eps d}{2}$, then for every $s_1 \leq s \leq s_0$, there exists an $(\ell,s,t)$-mop in $G$ with $t \geq \left(1 + \frac{\gamma \eps}{2K}\right)^{s-1}\frac{\eps d}{2}$.
	\end{claim*}
	\begin{cla_proof}
		We induct on $s$, where the base case $s=s_1$ holds by assumption. Now let $s_1 + 1 \leq s \leq s_0$, and consider an $(\ell,s-1,t')$-mop $M$ with $t' \geq \left(1 + \frac{\gamma \eps}{2K}\right)^{s-2}\frac{\eps d}{2}$. If $t' \geq \left(1 + \frac{\gamma \eps}{2K}\right)^{s-1}\frac{\eps d}{2}$, then $M$ is an $(\ell, s, t')$-mop satisfying the claim, so we may assume that $t' < \left(1 + \frac{\gamma \eps}{2K}\right)^{s-1}\frac{\eps d}{2} \leq \alpha^{-1}(d - \ell)$ since $\alpha \ll 1/s_0, \gamma, \eps, 1/K$ and $\ell \leq (1-\eps)d$. Since $G$ is a $(d,s_0, \gamma, \alpha)$-leaky mop expander, it follows that $M$ is $(d,\gamma)$-leaky. Now by \Cref{lem:more ends}, there exists an $(\ell,s,t)$-mop with
		\begin{equation}\label{eq:mop-growth}
		t \geq \left(1 + \frac{\gamma(d-\ell) - 2s_0}{Kd}\right)t' - (\ell + 1) \geq \left(1 + \frac{3\gamma \eps d}{4Kd}\right)t' - d \geq \left(1 + \frac{\gamma \eps}{2K}\right)^{s-1}\frac{\eps d}{2} + \frac{\gamma \eps}{4K}t' - d
		\end{equation}
		since $1/d \ll 1/s_0, \gamma, \eps, 1/K$ and $\ell \leq (1-\eps)d$. Now
		\[\frac{\gamma \eps}{4K}t' \geq \frac{\gamma \eps}{4K}\left(1 + \frac{\gamma \eps}{2K}\right)^{s_1-1}\frac{\eps d}{2} \geq d,\]
		since $1/s_1 \ll \gamma, \eps, 1/K$, so \eqref{eq:mop-growth} gives $t \geq \left(1 + \frac{\gamma \eps}{2K}\right)^{s-1}\frac{\eps d}{2}$. This concludes the inductive step.
	\end{cla_proof}
    
	The previous two claims show that we can grow the number of ends rapidly in exchange for increasing the length of the strands. The next ``extend'' claim is the key part of the proof; it shows that if we have a \emph{very} large number of ends, then we can increase the length of the handle at the expense of decreasing the number of ends (which we can then again boost by the above claim). By iterating this alternately with the boost claim, we will be able to construct mops with longer and longer handles until we reach a path of length at least $(1-\eps)d$. We will achieve this by simply appending a carefully chosen strand to the current handle, and showing that, thanks to our careful choice of the extending strand, this new handle also admits many new strands. 
    The major issue here, compared to the boost claim, is that we will need to grow our new ends set from scratch, so a long handle
    % our new handle is very long, and we are starting to build our new mop from scratch, so the handle
    could easily block the growth. To deal with this, we use our maximum degree assumption to argue that at least one of the ends  is ``far'' enough from the handle that many short walks from this end miss the handle completely.
    % we pick our strand very carefully, making sure its end (so the new end of our handle) is ``far'' from any\footnote{In reality, we can't ensure ``any'' here and need to work with a certain hereditary notion of ``most''.} vertices which are blocked by the handle too much.
    
    % The following ``extend'' claim is the key part of the proof, as it will allow us to construct mops with longer and longer handles (when iterated alternately with the boost claim) until we reach a path of length at least $(1-\eps)d$.
	\begin{claim*}[Extend]
		Let $0 \leq \ell \leq (1-\eps)d$. If there exists an $(\ell, s_0, t_0)$-mop in $G$ with $t_0 \geq \left(1 + \frac{\gamma \eps}{2K}\right)^{s_0-1}\frac{\eps d}{2}$, then there exists an $(\ell', s_1, t_1)$-mop with $\ell' > \ell$ and $t_1 \geq \left(1 + \frac{\gamma \eps}{2K}\right)^{s_1-1}\frac{\eps d}{2}$.
	\end{claim*}
	\begin{cla_proof}
		Let $(P, v, U, \{Q_u\}_{u \in U})$ be an $(\ell, s_0, t_0)$-mop. Let $C' := C(G) \setminus C(P)$, and let $d' := d - \ell$, which is at least $\eps d$ since $\ell \leq (1-\eps)d$. Define $B_0 := V(P)$, and for each $1 \leq i \leq s_1$, define
    \[B_i := B_{i-1} \cup \{x \in V(G) : \deg(x, B_{i-1}; C') \geq \gamma d'/4\}.\]
    One should think of these $B_i$'s as the set of ``bad/blocked'' vertices at distance $i$. Namely, $B_0$ consists of vertices already on the handle, and each subsequent $B_i$ consists of vertices with ``too many'' neighbours in the previous bad set $B_{i-1}$ (which includes the handle). Our first task is to bound the growth of the bad sets to show that the largest set $B_{s_1}$ cannot contain all of the ends in $U$.
    % to show that the number of bad vertices does not grow too fast so that we can pick an end $u \in U$ which is not in $B_{s_{1}}$, which we will use to extend our handle.
    % \footnote{To help with intuition, the reader might think of an idealised case in which $u \notin B_{s_{1}}$ means $u$ has no neighbours in $B_{s_1}$ (rather than just having ``few'' of them). The benefit of this recursive definition now is that any vertex at distance up to $s_1$ from $u$ does not have any neighbours on $P$ and can't be ``blocked'' by our current, very long, handle.}. 
    See \Cref{fig:extend-handle} for an illustration. 

    \begin{figure}[h]
    \centering
    \includegraphics[width=0.7\linewidth]{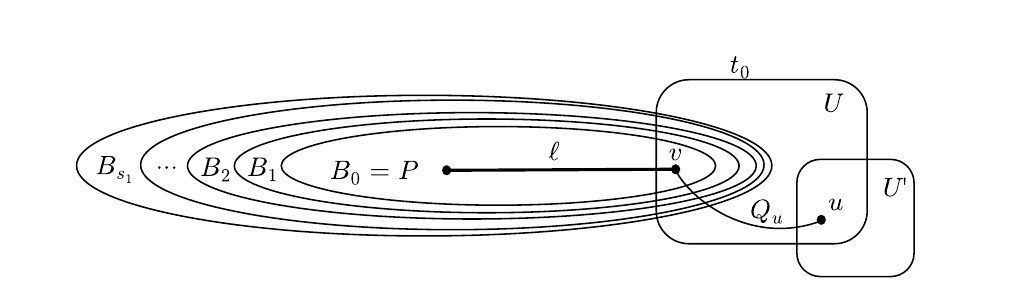}
    \caption{Illustration of the argument in the extend claim. We iteratively build sets $B_i$ which have a lot of neighbours in the previous set $B_{i-1}$ (starting with $B_0=P$). By choosing an end $u\in  U\setminus B_{s_1}$, we can extend the handle to $P \cup Q_u$ and build new strands from scratch using \Cref{lem:more ends bad set} so as to avoid bad sets (when building layer $i$ we avoid $B_{s_1-i}$).}
    \label{fig:extend-handle}
\end{figure}
    
    To do so, note that since $\Delta(G) \leq Kd \leq \eps^{-1} K d'$, we have $|B_i| \leq |B_{i-1}|+|B_{i-1}|\Delta(G)/(\gamma d'/4) \leq (1 + 4\gamma^{-1}\eps^{-1} K)|B_{i-1}|$. In particular,
    \[|B_{s_1}| \leq \left(1 + 4\gamma^{-1}\eps^{-1} K\right)^{s_1}d < \left(1 + \frac{\gamma \eps}{2K}\right)^{s_0-1}\eps d/2 \leq t = |U|\]
    since $1/s_0 \ll 1/s_1, \gamma, \eps, 1/K$,
    so there indeed exists some vertex $u \in U \setminus B_{s_1}$. Let $P' := P \cup Q_u$, and let $\ell' := \ell + e(Q_u) > \ell$ be the length of $P'$. Our new mop will have $P'$ as a handle and $u$ as its end-vertex, and what remains to be shown is that we can still find many ends using strands of length up to $s_1$.
    % We will grow this mop iteratively from scratch. The issue here is that we already have a large handle to begin with, so we might get stuck if too many edges from the current ends go back to the handle. This is where $u \notin B_{s_1}$ comes to the rescue. Namely, this guarantees that $u$ doesn't have many neighbours on the handle, in fact, it doesn't have many neighbours in $B_{s_1-1}$, so the first batch of ends we produce can avoid this set and hence inherit the same property w.r.t.\ $B_{s_1-2}$. So we can iteratively construct strands of length up to $i$ with few neighbours in $B_{s_1-i-1}$, allowing us to iteratively build our new mop. To grow the mop, we will make use of \Cref{lem:more ends bad set}, using its additional feature, which allows us to avoid the next bad set. 
		
	We will produce, for each $1 \leq i \leq s_1$, in turn, an $(\ell',i, t_i')$-mop $M_i = (P', u, W_i, \{\tilde Q_{i,w}\}_{w \in W_i})$ with $t_i' \geq \left(1 + \frac{\gamma \eps}{2K}\right)^{i-1}\frac{\eps d}{2}$ and $W_i \cap B_{s_1-i} = \emptyset$, as follows.
		
		First, since $G$ is properly edge-coloured with minimum degree $d$, $G[C']$ has minimum degree at least $d'$. Thus $u$ has at least $d'$ neighbours $w$ with $C(uw) \in C'$. Since $u \notin B_{s_1}$, at most $\gamma d'/4$ of these are in $B_{s_1-1}$, at most $|C(Q_u)| \leq s_1$ have $C(uw) \in C(Q_u)$, and at most $|V(Q_u) \setminus \{u\}|$ have $w \in V(Q_u) \setminus \{u\}$. Since $V(P) \subseteq B_{s_1-1}$, the remaining neighbours $w$ are not in $V(P')$, so $P' \cup \{uw\}$ is a rainbow path. Let $W_1$ be the set of such neighbours. By what we have observed, this has size
		\[|W_1| \geq d' - \frac{\gamma d'}{4} - 2s_0 \geq \frac{\eps d}{2}\]
		since $1/d \ll 1/s_0, \eps$ and $d' \geq \eps d$. Thus $M_1 := (P', u, W_1, \{uw\}_{w \in W})$ is an $(\ell', 1, t_1')$-mop with the desired properties.
		
		Now suppose we have constructed our desired mop $M_{i-1} = (P', u, W_{i-1}, \{\tilde Q_{i-1,w}\}_{w \in W_{i-1}})$ for some $2 \leq i \leq s_1$. If $t_{i-1}' \geq \left(1 + \frac{\gamma \eps}{2K}\right)^{i-1}\frac{\eps d}{2}$, then we can take $M_i = M_{i-1}$. Otherwise, we have $t_{i-1}' < \left(1 + \frac{\gamma \eps}{2K}\right)^{s_1-1}\frac{\eps d}{2} \leq \alpha^{-1} (d-\ell')$ since $1/d, \alpha \ll 1/s_1, \gamma, \eps, 1/K$ and $d - \ell' \geq d' - s_1 \geq \frac{\eps d}{2}$. Therefore, since $G$ is a $(d,s_0, \gamma, \alpha)$-leaky mop expander, $M_{i-1}$ is $(d,\gamma)$-leaky. Also since $W_{i-1} \cap B_{s_1-i+1} = \emptyset$, we have that 
		\[\deg(w,B_{s_1-i};C(G) \setminus C(P')) \leq \deg(w,B_{s_1-i};C') \leq \gamma d'/4\]
		for every $w \in W_{i-1}$. Therefore, by \Cref{lem:more ends bad set} with $\ell$, $s$, $t$, $\Delta$, $\Delta'$, $(P, v, U, \{Q_u\}_{u \in U})$, $B$ replaced with $\ell'$, $i-1$, $t_{i-1}'$, $Kd$, $\gamma d'/4$, $(P', u, W_{i-1}, \{\tilde Q_{i-1,w}\}_{w \in W_{i-1}})$, $B_{s_1-i}$, respectively, there exists an $(\ell', i, t_i')$-mop $M_i := (P', u, W_i, \{\tilde Q_{i,w}\}_{w \in W_i})$ with $W_i \cap B_{s_1-i} = \emptyset$ and 
		\begin{align*}
			t_i' &\geq \left(1 + \frac{\gamma(d-\ell') - \gamma d'/4 - 2s_1}{Kd}\right)t_{i-1}' - |V(P') \setminus B_{s_1-i}| \\
			&\geq \left(1 + \frac{3\gamma \eps d/4 - \gamma s_0 - 2s_1}{Kd}\right)t_{i-1}' - |V(P') \setminus B_{s_1-i}| \\
			&\geq \left(1 + \frac{5\gamma \eps}{8K}\right)t_{i-1}' - |V(P') \setminus B_{s_1-i}|
		\end{align*}
		since $\ell' - \ell \leq s_0$, $d' = d- \ell \geq \eps d$, and $1/d \ll 1/s_0, 1/s_1, \gamma, \eps$. Now $V(P) \subseteq B_{s_1-i}$, so $|V(P') \setminus B_{s_1-i}| \leq |V(Q_u) \setminus \{v\}| \leq s_0$, and we have
		\[t_i' \geq \left(1 + \frac{\gamma \eps}{2K}\right)^{i-1} \frac{\eps d}{2} + \frac{\gamma \eps}{8K}t_{i-1}' - s_0 \geq \left(1 + \frac{\gamma \eps}{2K}\right)^{i-1} \frac{\eps d}{2}.\]
		since $t_{i-1}' \geq \left(1 + \frac{\gamma \eps}{2K}\right)^{i-2} \frac{\eps d}{2} \geq \frac{\eps d}{2}$ and $1/d \ll 1/s_0, \gamma, \eps, 1/K$. Thus $M_i$ satisfies the desired properties.
		
		By induction, we have that $M_{s_1}$ is an $(\ell', s_1, t_{s_1}')$-mop with $t_{s_1}' \geq \left(1 + \frac{\gamma \eps}{2K}\right)^{s_1-1} \frac{\eps d}{2}$, and this proves the claim.
	\end{cla_proof}
	Now to finish the proof, we consider an $(\ell,s_1,t)$-mop $(P, v, U, \{Q_u\}_{u \in U})$ with $t \geq \left(1 + \frac{\gamma \eps}{2K}\right)^{s_1-1} \frac{\eps d}{2}$, and with $\ell$ as large as possible. By the initialise claim, such mops exist with $\ell = 0$, and by the boost and extend claims and the maximality of $\ell$, we must have $\ell > (1 - \eps)d$; that is, $P$ is a path of length greater than $(1-\eps)d$.
\end{proof}
		
\subsection{Non-expanding case}\label{sec:non-expanding case}

We are now ready to prove \Cref{thm:schrijver-asymptotic-intro}. In fact, we prove the following slightly more general result for almost-regular graphs.
By \Cref{lem:mop argument}, we only need to deal with the case when $G$ is \emph{not} a leaky mop expander. This means it contains a mop that fails to be leaky, so that the set of ends of this mop does not send many edges outside in non-handle colours, and we may find a robustly expanding subgraph inside it. We may then use the results of Sections~\ref{sec:pathforests} and~\ref{sec:expansion} to find a rainbow path within this robust expander, which combines with the handle of the mop and one of the strands to give the desired rainbow path. 

\begin{theorem} \label{thm:main}
    For any $\eps, K > 0$, there exists $d_0$ such that the following holds for all $d \geq d_0$. Let $G$ be a properly edge-coloured graph with $d = \delta(G) \leq \Delta(G) \leq Kd$. Then $G$ has a rainbow path of length at least $(1-\eps)d$.
\end{theorem}

\begin{proof}
We may assume that $\eps, 1/K \ll 1$. Let $1/d \ll \nu \ll \tau \ll \alpha \ll 1/s_0 \ll \gamma \ll \eps, 1/K \ll 1$. 

Suppose towards a contradiction that $G$ contains no rainbow path of length at least $(1-\eps)d$. By \Cref{lem:mop argument}, $G$ is not a $(d, s_0, \gamma, \alpha)$-leaky mop expander.

The following claim takes a non-leaky mop in $G$ and applies our pass-to-expander lemma (\Cref{lem:find robust expander}) inside of its set of ends. Technically, in order to be able to apply it and to maintain the failure of expansion (our expander might be substantially smaller than the initial set) some care is needed, and this is where the additional technical property~\ref{findrobustexpanderdense:3} will be necessary. We include \Cref{fig:non-expanding} to help with following the argument.

\begin{claim*}
    There exists an $(\ell, s, n')$-mop $(P, v, U, \{Q_u\}_{u \in U})$ satisfying the following properties with $d':=d-\ell$ and $C':=C(G) \setminus C(P)$:
    \stepcounter{propcounter}
    \begin{enumerate}[label = {{{\normalfont\textbf{\Alph{propcounter}\arabic{enumi}}}}}]
        \item 
        $\ell \leq (1-\eps)d$, $s \leq s_0$, and $n' = |U| \leq \alpha^{-1}d'$;
        \label{notleakycase:1}
        \item $e(U,U^c; C') < 3\gamma d'n'$;
        \label{notleakycase:2}
        \item the graph $G' := G[U; C']$ is a $(\nu, \tau)$-robust expander;
        \label{notleakycase:3}
        \item $\delta(G') \geq (1-3\gamma)d'/2$.
        \label{notleakycase:4}
    \end{enumerate}
\end{claim*}

\begin{cla_proof}
    Since $G$ is not a $(d, s_0, \gamma, \alpha)$-leaky mop expander, there exists an $(\ell,s,t)$-mop $(P, v, U_0, \{Q_u\}_{u \in U_0})$ with $s \leq s_0$, $t=|U_0| \leq \alpha^{-1}d'$, and $e(U_0,U_0^c; C') < \gamma d'|U_0|$, where $C' := C(G) \setminus C(P)$ and $d' := d-\ell$. 
    We have assumed that $G$ has no rainbow path of length at least $(1-\eps)d$, so we have $\ell \leq (1-\eps)d$ as well. 
    
    Since we only have control of the number of edges with colours in $C'$ leaving $U_0$, we currently only control the average degree in $G[U_0;C']$, and in order to apply our pass-to-expander lemma (\Cref{lem:find robust expander}) we need control on the minimum degree. So we first perform a ``clean-up'' step by passing into a minimal subset $U_1 \subseteq U_0$ with $e(U_1, U_1^c; C') < \gamma d'|U_1|$. 
    This guarantees that 
    \[\delta(G[U_1;C'])\geq (1-\gamma)d'/2 \geq \alpha t/4 \geq \alpha |U_1|/4.\] Indeed, if there existed a vertex $u \in U_1$ whose degree in the graph $G[U_1;C']$ was less than  $(1-\gamma)d'/2$, then $W:=U_1\setminus \{u\}$ would be a smaller set with the same property, indeed in this case $W$ satisfies
    \begin{align*} e(W, W^c; C')&=e(U_1, U_1^c;C')-\deg(u,U_1^c;C')+\deg(u,U_1;C') \\
    &=e(U_1, U_1^c;C')-\deg(u;C')+2\deg(u,U_1;C') \\
    &< \gamma d'|U_1|-d'+(1-\gamma)d'\\
    &=\gamma d'|W|,    
    \end{align*}
    which would be a contradiction to the minimality of $U_1$. Now, recalling that $\nu \ll \tau, \alpha, \gamma$, we can apply \Cref{lem:find robust expander} to the symmetric digraph associated to $G[U_1; C']$ with $\gamma \alpha$, $\alpha/4$ 
    in place of $\delta, \alpha$ and with $w(v) := \deg(v,U_1^c; C')$ to obtain a nonempty induced subgraph $G'\subseteq G[U_1; C']$ which satisfies \ref{findrobustexpanderdense:1}-\ref{findrobustexpanderdense:3}; that is, $G'$ is a robust $(\nu, \tau)$-expander, every vertex $v\in U:=V(G')$ satisfies \begin{equation}\label{eq:expander degree loss}
    \deg_{G'}(v)\geq \deg_{G[U_1;C']}(v) - \gamma \alpha|U_1| \geq (1-\gamma)d'/2 - \gamma d' = (1-3\gamma)d'/2,
    \end{equation}
    and
    \begin{equation}\label{eq:relative-weight}
        e(U, U_1^c; C') = \sum_{v \in U} w(v) \leq 2\frac{|U|}{|U_1|}\sum_{v \in U_1} w(v) = 2\frac{|U|}{|U_1|}e(U_1, U_1^c; C')< 2\gamma d' |U|.
    \end{equation}

    Now, defining $n' := |U| \leq t$, the $(\ell,s,n')$-mop $(P, v, U, \{Q_u\}_{u \in U})$ satisfies \ref{notleakycase:1}, \ref{notleakycase:3}, and \ref{notleakycase:4}. We now verify \ref{notleakycase:2}. It will follow since \eqref{eq:expander degree loss} guarantees that vertices of $G'$ have very similar degrees as they did in $G[U_1; C']$, and the property~\ref{findrobustexpanderdense:3} with our choice of weight function gave us control over the number of edges escaping to $U_1^c$, as seen in \eqref{eq:relative-weight}.
    Indeed, any vertex $u\in U$ sends at most $\deg_{G[U_1;C']}(u)-\deg_{G'}(u)\leq \gamma \alpha|U_1|\leq \gamma d'$ edges to $U_1\setminus U$ in $C'$ colours by \eqref{eq:expander degree loss}, so combined with \eqref{eq:relative-weight} we have
    $$e(U, U^c;C') = e(U, U_1^c; C') + e(U, U_1 \setminus U;C') < 2\gamma d'|U|+\gamma d'|U| = 3\gamma d'n'.$$
    Thus \ref{notleakycase:2} holds as well. This proves the claim.
\end{cla_proof}

\vspace{-0.3cm}
\begin{figure}[h]
    \centering
    \includegraphics[width=0.7\linewidth]{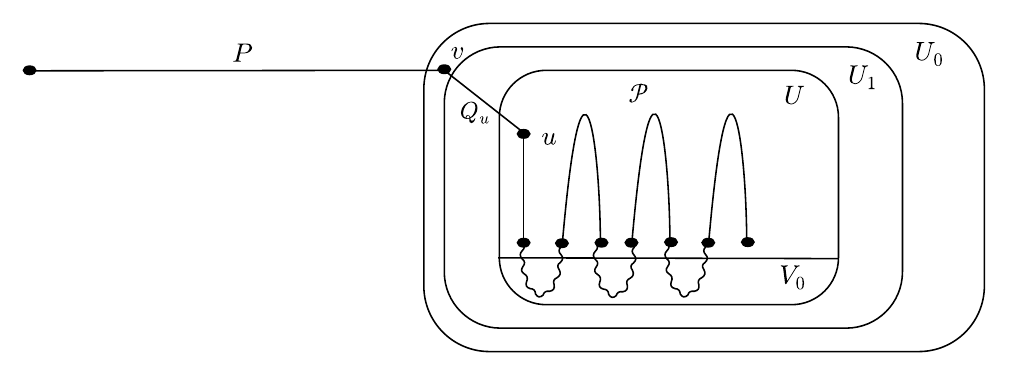}
    \caption{Illustration of the argument in the non-expanding case. $U_0$ is the initial set of ends of a mop witnessing $G$ not being a suitable leaky mop expander. This guarantees the graph $G[U_0;C']$ has average degree close to $d'$. $U_1$ is a slightly smaller subset in which we regain some control on the minimum degree and $U \subseteq U_1$ is then obtained by applying the pass-to-expander lemma (\Cref{lem:find robust expander}) to $G[U_1;C']$. We then reserve random robustly connected subsets $V_0$ of $U$ and $C_0$ of $C$ (using expansion by means of \Cref{lem:connecting lemma}), find a linear path forest not using these vertices or colours with size close to $d'$, and finally use short paths through $V_0$ using colours from $C_0$ to link up the paths.}
    \label{fig:non-expanding}
\end{figure}

A keen-eyed reader might wonder why we are not just able to apply \Cref{lem:long path in robust expander} to $G'$ to find a rainbow path of length $(1-\eps)d'$ which we can append to the handle and one of the strands to get our desired rainbow path. There are two main reasons for this. First, we need to take care to make sure that the new path does not reuse vertices or colours from the connecting strand. Second, \Cref{lem:long path in robust expander} will only give us a rainbow path of length $(1-\eps)\delta(G')$, and we have only managed to maintain the minimum degree in~$G'$ up to roughly a factor of~$1/2$ from $d'$. To deal with both of these issues, we will instead retrace the proof of \Cref{lem:long path in robust expander} by hand: randomly sampling our robust expander $G'$ to create a connecting reservoir, finding a large rainbow path forest with few components, then linking up the components through the reservoir in such a way so as to avoid the vertices and colours of a designated strand. To overcome the minimum degree deficiency when finding our rainbow path forest, we make use of both rainbow path forest lemmas from \Cref{sec:pathforests}. The first (\Cref{lem:sparsepathforest}) can only be applied if $|V(G')|$ is much larger than its maximum degree, but in this case, it gives a rainbow path forest of size close to the \emph{average} degree of $G'$. The second (\Cref{lem:rainbow path forest}) has no such maximum degree requirement, but it only gives a rainbow path forest of size close to the \emph{minimum} degree of $G'$. Fortunately, the property~\ref{notleakycase:2} implies that the average degree of $G'$ is still close to $d'$. If $|V(G')|$ is large enough, then we can apply \Cref{lem:sparsepathforest} directly, and if not, then we can boost the minimum degree by passing to a further subgraph of $G'$, then apply \Cref{lem:rainbow path forest} to obtain our rainbow path forest.
% The main reason is that we are only able to maintain average degree in $G'$ rather than minimum degree as needed for the argument there.  To deal with this, we will make use of \Cref{lem:sparsepathforest} to find a rainbow path forest with few components. Note that crucially, this lemma does not have a requirement on the minimum degree, but it does have one on the maximum degree, so we will need to combine it with our second rainbow path forest lemma (\Cref{lem:rainbow path forest}). Finally, we will link up the paths in the path forest using the expansion properties of $G'$ (by means of \Cref{lem:connecting lemma}). To ensure the short linking paths do not intersect with paths already in our rainbow path forest, we will reserve a random set of colours and vertices through which we will make the connections.

We now proceed with this plan by invoking \Cref{lem:connecting lemma} to sample our reservoir. Let $p := \gamma$ and $\beta := p^3 \nu/100$. Since $n' \geq (1-3\gamma)d'/2 \geq \eps d/4$ by \ref{notleakycase:1} and \ref{notleakycase:4}, 
and $1/d \ll \nu, \gamma, \eps$,
we have $p^3\nu^2 n' \geq 144 \log n'$.
We also have $\delta(G') \geq (1-3\gamma)d'/2 \geq \alpha n'/4$ by \ref{notleakycase:1} and \ref{notleakycase:4}, and $\nu + \tau \leq \alpha/4$ since $\nu \leq \tau \ll \alpha$, so we can apply \Cref{lem:connecting lemma} to the symmetric digraph associated to $G'$ with $\alpha/4$, $n'$ in place of $\alpha$, $n$, respectively, to find $V_0 \subseteq V(G')$ and $C_0 \subseteq C'$ satisfying \ref{connectprop:1}-\ref{connectprop:3}. Set $G'' := G'[V(G') \setminus V_0; C' \setminus C_0]$. 
Note that by \ref{connectprop:1} and \ref{connectprop:2}, 
\begin{equation}\label{eq:properties of G''}
    n'' := |V(G'')| \geq (1-2\gamma)n' \quad \quad \text{ and } \quad \quad \deg_{G''}(v) \geq (1-3\gamma)\deg_{G'}(v) \:\: \text{for every } v \in V(G'').
\end{equation}

\begin{claim*}
    There exists a rainbow path forest $\mathcal P$ in $G''$ with at most $144\eps^{-2}$ components and at least $(1-\eps/2)d'$ edges.
\end{claim*}
\begin{cla_proof}
    We consider two cases. If $Kd \leq \eps n''/16$, then $G''$ has maximum degree at most $\eps n''/16$ and average degree at least 
    \begin{align*}
        \frac{1}{n''}\sum_{v \in V(G'')} \deg_{G''}(v) &\geq \frac{1}{n''}\sum_{v \in V(G'')} (1-3\gamma)\deg_{G'}(v) \\
        &= \frac{1}{n''}\sum_{v \in V(G'')} (1-3\gamma) (\deg_{G}(v;C') - \deg_G(v, U^c; C')) \\
        &\geq (1-3\gamma)d' - \frac{(1-3\gamma)(3\gamma d'n')}{n''} \\
        &\geq (1-3\gamma)d' - \frac{(1-3\gamma)(3\gamma d')}{1-2\gamma} \geq (1-6\gamma)d'
    \end{align*}
    by \eqref{eq:properties of G''}, \ref{notleakycase:2}, and the fact that every vertex $v \in V(G)$ satisfies $\deg_G(v;C') \geq d - |C(P)| = d'$. Thus we can apply \Cref{lem:sparsepathforest} to $G''$ with $\eps/4$, $(1-6\gamma)d'$ in place of $\eps$, $d$ to obtain a rainbow path forest with at most $\lceil \log_{8/7} (\eps/4)^{-1} \rceil \leq 144\eps^{-2}$ components and at least $(1-\eps/4)(1-6\gamma)d' \geq (1-\eps/2)d'$ edges, as desired.

    On the other hand, if $K d \geq \eps n''/16$, then we can find a subgraph $G'''$ of $G''$ with minimum degree close to $d'$ as follows. Let $S$ be the set of vertices $u \in V(G'')$ with $\deg_G(u, U^c; C') \geq \sqrt{3\gamma d' n'}$. Then, because $V(G'') \subseteq U$ and $e_G(U,U^c; C') \leq 3\gamma d'n'$ by \ref{notleakycase:2}, we must have that $|S| \leq \sqrt{3\gamma d'n'}$. Recall that $n'' \geq (1-2p)n' \geq n'/2$ and $d' \geq \eps d$, so for $\gamma' := \sqrt{96\gamma \eps^{-2} K}$, we have 
    \[|S| \leq \sqrt{3\gamma d'n'} \leq \sqrt{6\gamma d' n''} \leq \gamma' d',\]
    where in the final inequality we used $Kd \geq \eps n''/16$.

    Also note that $\gamma \leq \gamma' \ll \eps$ since we chose $\gamma \ll \eps, 1/K \ll 1$. Deleting vertices of $S$ from $V(G'')$, we obtain $G''':=G''\setminus S$ with minimum degree at least $(1-5\gamma')d'$. Indeed, for any vertex $v\in V(G''')$, we have that 
    \[\deg_{G'}(v) = \deg_G(v;C') - \deg_G(v, U^c; C') \geq d' - \sqrt{3\gamma d' n'} \geq (1-\gamma')d'\]
    since $v \notin S$, so by \eqref{eq:properties of G''}, we are sure to have 
     $\deg_{G''}(v)\geq (1-3\gamma)(1-\gamma')d'\geq (1-4\gamma')d'$ since $\gamma \leq \gamma'$. Finally, $v$ can send at most $|S| \leq \gamma'd'$ edges to $S$, so we are left with $\deg_{G'''}(v)\geq (1-5\gamma')d'$.
     
     Now $\delta(G''') \geq (1-5\gamma')d' \geq 9(\eps/4)^{-3}$ since $d' \geq \eps d$, $1/d \ll \eps$, and $\gamma' \ll 1$. Therefore, we can apply \Cref{lem:rainbow path forest} to the symmetric digraph associated to $G'''$ with $\eps/4$ in place of $\eps$ to obtain a rainbow path forest $\mathcal{P}$ with at most $9(\eps/4)^{-2} = 144\eps^{-2}$ components and at least $(1-\eps/4)(1-5\gamma')d'\geq (1-\eps/2)d'$ edges since $\gamma' \ll \eps$, as required. This proves the claim.
\end{cla_proof}

Now that we have our reservoir sets $V_0$ and $C_0$ and our rainbow path forest $\mathcal P$, we need to make sure that after connecting the components, the path we obtain has an endpoint whose strand does not use the same vertices or colours as the path.
% our path here avoids the vertices and colours of the strand connecting its start-vertex to the handle.
We deal with this by simply designating the target endpoint now and removing the colours and vertices of its strand from $\mathcal P$. We will also need to use the robust connectivity property~\ref{connectprop:3} of the reservoir to join the remaining components while avoiding the strand.

So let us fix a vertex $u$ from the endpoints of $\mathcal{P}$. By deleting all edges with colours in $C(Q_u)$ and all vertices in $V(Q_u) \setminus \{u\}$ from $\mathcal{P}$, we obtain a rainbow path forest $\mathcal{P}'$ whose colours are disjoint from $C(P\cup Q_u)$, with $V(\mathcal{P'})\cap V(P\cup Q_u)= \{u\}$. Note that each vertex or colour we delete from $\mathcal P$ creates at most one new component and deletes at most two edges. Therefore, since $Q_u$ has length at most $s_0$, $d' \geq \eps d$, and $1/d \ll 1/s_0 \ll \eps$, we see that $\mathcal P'$ has at most $144 \eps^{-2} + |C(Q_u)| + |V(Q_u)|-1 \leq 3s_0$ components and at least $(1-\eps/2)d' - 2(|C(Q_u)| + |V(Q_u)|-1) \geq (1-\eps)d'$ edges.

So far, we have a rainbow path $P \cup Q_u$ of length at least $\ell$ and a rainbow path forest $\mathcal P'$ with at least $(1-\eps)d'$ edges which are colour-disjoint and intersect only at the shared endpoint $u$. We now join the components of $\mathcal P'$ via short, internal-vertex-disjoint, colour-disjoint, rainbow paths using vertices from $V_0$ and colours from $C_0$ to obtain a single rainbow path of length at least $\ell + (1-\eps)d' \geq (1-\eps)d$. Note that the only vertices and colours of $P \cup Q_u \cup \mathcal P'$ that might appear in $V_0$ or $C_0$ come from the short path $Q_u$, so avoiding these when finding these connecting paths is not too difficult by \ref{connectprop:3}. We denote the paths of $\mathcal{P'}$ to be $P_1, P_2, \dots, P_m$, where $m\leq 3s_0$, and $P_i$ joins vertices $u_i$ and $v_i$, with $u_1=u$. Denote $Q_0:=Q_u$. For each $i = 1, \ldots, m-1$, in turn, we iteratively apply property \ref{connectprop:3} of \Cref{lem:connecting lemma} to $v_i$ and $u_{i+1}$, with $V_i:=V_0 \cap \cup_{j\leq i-1}V(Q_j)$ and $C_i:=C_0 \cap \cup_{j\leq i-1}C(Q_j)$ playing the respective roles of $V_1$ and $C_1$, to obtain a rainbow path $Q_i$ connecting $v_i$ and $u_{i+1}$ of length at most $\nu^{-1}+1$ whose internal vertices are in $V_0\setminus V_i$ and whose colours are in $C_0\setminus C_i$. Note that at every step, we have $|V_i|, |C_i|\leq s_0+(\nu^{-1}+1)(m-1) \leq \beta  n'$, since $\beta=\gamma^3\nu/100$, $n' \geq \delta(G') \geq (1-3\gamma)d'/2 \geq \eps d/4$, and $1/d \ll \nu, s_0, \gamma, \eps$, so the iterative application is valid. At the end we obtain a rainbow path $P\cup Q_u\cup \mathcal P' \cup \ \bigcup_{i=1}^{m-1}Q_i$ of length at least $(1-\eps)d$.
\end{proof}

\section{The rearrangement problem in general groups}\label{sec:rearrangement}

In this section, given a Cayley graph $G = \mathrm{Cay}(\Gamma,S)$ with $|S| = d$ sufficiently large, we show how to build a directed rainbow walk of length $d$ with at most $\eps d$ vertex repetitions, thereby proving Theorem~\ref{thm:weakasymptotic-intro} (restated as Theorem~\ref{thm:weakasymptoticrestated} below). We achieve this incrementally, in $O(\eps^{-1})$ steps. At a given stage $i$ in our algorithm, we have a directed rainbow walk $P_{i-1}$ and a set $S_{i-1}$ of unused colours. If the graph induced by the colour set $S_{i-1}$ has good expansion properties, then we can reach a large set $X_i$ of vertices by short (length $O(\eps^{-1})$) $S_{i-1}$-rainbow paths from the terminal vertex of $P_{i-1}$, similar to the proof of \Cref{thm:schrijver-asymptotic-intro} in \Cref{sec:mop}. We can further ensure that all of these short paths use a common colour set $C_i$ of size $o(d)$. By \Cref{lem:GreedyOnePath}, we can find a rainbow path $Q_i$ of length $\Omega(\eps d)$ in the remaining colours (possibly reusing vertices of $P_{i-1}$). By a simple double-counting argument, there exists a translate of $Q_i$ which starts at a vertex in $X_i$ and intersects $P_{i-1}$ in $o(|Q|)$ vertices. We can thus append $Q_i$ to $P_{i-1}$ to obtain the new walk $P_i$, and we repeat. If, at some stage, we don't have the necessary expansion properties to continue, then we can find a set $U$ of vertices in which $G[U;S_{i-1}]$ satisfies the hypotheses of \Cref{theorem:rainbowdirectedpathindensecase}. At this point, we can find an $S_{i-1}$-rainbow path $Q$ of length at least $(1-o(1))|S_{i-1}|$, which we can similarly translate to begin in the set $X_{i-1}$ from the previous iteration so that it only intersects $P_{i-2}$ in $o(|Q|)$ vertices. We also take care to ensure that $Q$ does not use colours from the set $C_{i-1}$ reserved for joining $P_{i-2}$ to $X_{i-1}$, so we can append $Q$ to $P_{i-2}$ to obtain a rainbow walk of length at least $(1-\eps/2)d$ with at most $\eps d/2$ vertex repetitions. Now, an arbitrary extension in the remaining $\leq \eps d/2$ colours gives the desired rainbow walk of length $d$.

\begin{theorem}\label{thm:weakasymptoticrestated}
    For every $\eps>0$, there exists $d_0\in\mathbb{N}$ such that for all groups $\Gamma$ and all subsets $S\subseteq \Gamma$ with $|S|=d\geq d_0$, $\mathrm{Cay}(\Gamma, S)$ contains a directed rainbow walk spanning $d$ edges with at most $\eps d$ vertex repetitions.
\end{theorem}

\begin{proof}

We choose constants $\gamma', \gamma, \alpha, \theta, \delta$ and $\ell,k$ so that $1/n \leq 1/d \ll 1/\ell  \ll  \gamma'\ll \gamma \ll \alpha \ll \theta,  1/k,  \delta \ll \eps$. Let $S$ be a subset of size $d$ of a group $\Gamma$ of order $n$, and let $G := \mathrm{Cay}(\Gamma,S)$. Since we are aiming for an asymptotic result, we may assume that $S$ does not contain the identity element. 

Before we begin, we note that it suffices to find a rainbow directed walk $P$ of length at least $(1-\eps/2)d$ with at most $\eps d/2$ vertex repetitions. Then, because every vertex has an out-neighbour in every colour from $S$, we can greedily extend $P$ to a rainbow walk of length $d$ with at most $\eps d$ vertex repetitions. 

\par \textit{\textbf{The strategy.}} We attempt to iteratively create rainbow walks $P_0\subseteq P_1\subseteq\cdots$ such that each $P_{i-1}$ is the initial segment of $P_i$. We also ensure that for each $i \geq 0$, we have $\delta d \leq |C(P_i)| - |C(P_{i-1})| \leq \lceil \delta d \rceil + \ell$
and $|V(P_i)| \geq |C(P_i)| -  2i\delta d/k - i\ell + 1$. We terminate this process once we have $|C(P_{i})| \geq (1-\eps/2)d$, unless our procedure fails before that point. Thus we will always have $i \leq \delta^{-1}$. At step $i=0$, let $v_0$ be an arbitrary vertex of $V(G)$ and $P_0$ be an empty walk with a single vertex $v_0$. For all $i$, $v_0$ will be the initial vertex of $P_i$, and we will denote the terminal vertex of $P_i$ by $v_i$. Suppose we are at step $i$ and that we have constructed $P_0 \subseteq \cdots \subseteq P_{i-1}$. If we have that $|C(P_{i-1})| \geq (1-\eps/2)d$, then we are already done with $P = P_{i-1}$, as the number of vertex repetitions of $P_{i-1}$ is only
\[|C(P_{i-1})| + 1 - |V(P_{i-1})| \leq 2d/k + \delta^{-1} \ell \leq \eps d/2,\]
since $1/d \ll 1/\ell, 1/k, \delta \ll \eps$. Otherwise, we have $|S \setminus C(P_{i-1})| \geq \eps d/2$, and we attempt to construct $P_i$ as follows.

\par \textit{\textbf{The inductive step.}} We say that a colour set $C \subseteq S \setminus C(P_{i-1})$ is $(\theta,k,\ell)$-\emph{expanding} if $|C| \leq \theta d$, and there exist at least $kd$ distinct vertices which are the terminal vertex of some $C$-rainbow walk of length at most $\ell$ from $v_{i-1}$.

\par \textit{\textbf{Case 1: There exists an expanding colour set.}} If there exists a $(\theta,k,\ell)$-expanding colour set $C_i \subseteq S \setminus C(P_{i-1})$, let $X_i$ be the set of at least $kd$ vertices guaranteed by the definition of $(\theta,k,\ell)$-expanding. Because $|S \setminus (C(P_{i-1}) \cup C_i)| \geq \eps d/2 - \theta d \geq 2\lceil \delta d \rceil$ 
since $\delta, \theta \ll \eps$, we can, by \Cref{lem:GreedyPathMinDegree}, find a rainbow path $Q_i$ of length $\lceil \delta d \rceil$ starting at the identity element $e$ using colours from $S \setminus (C(P_{i-1}) \cup C_i)$.  

Next, we will show that we can translate this path so that it starts at some vertex in $X_i$ without intersecting $P_{i-1}$ too much. 

\begin{claim*} There exists $x_i\in X_i$ such that $|x_iV(Q_i)\cap V(P_{i-1})| \leq 2\delta d/k$. 
\end{claim*}

\begin{cla_proof}
We count the number $s$ of triples $(x,u,v) \in X_i \times V(Q_i) \times V(P_{i-1})$ such that $xu = v$ in $\Gamma$. Clearly 
\[s \leq |V(Q_i)| |V(P_{i-1})| \leq 2\delta d^2\]
since $|V(Q_i)| \leq 1 + \lceil \delta d \rceil \leq 2\delta d$ and $|V(P_{i-1})| \leq d$. On the other hand, if the claim does not hold, then 
\[s > (2\delta d/k)|X_i| \geq 2\delta d^2\]
since $|X_i| \geq kd$, so we have a contradiction.
\end{cla_proof}

Let $x_i$ be as in the claim, and let $L_i$ be a $C_i$-rainbow walk of length at most $\ell$ from $v_{i-1}$ to $x_i$. We take $P_i$ to be $P_{i-1}\cup L_i\cup (x_iQ_i)$. By design, $P_{i-1}$, $L_i$, and $x_iQ_i$ are colour-disjoint, so $P_i$ is a rainbow walk with initial segment $P_{i-1}$. Moreover, $|C(P_i)| - |C(P_{i-1})| = |C(L_i)| + |C(Q_i)|$ is between $\delta d$ and $\lceil \delta d \rceil + \ell$, and
\begin{align*}
    |V(P_i)| &\geq |V(P_{i-1})| + |V(Q_i)| - |x_iV(Q_i) \cap V(P_{i-1})| \\
    &\geq |C(P_{i-1})| - 2(i-1)\delta d/k - (i-1)\ell + 1 + |C(Q_i)| - 2\delta d/k \\
    &= |C(P_i)| - |C(L_i)| - 2(i-1)\delta d/k - (i-1)\ell + 1 - 2\delta d/k \\
    &\geq |C(P_i)| - 2i\delta d/k - i \ell + 1,
\end{align*}
as required.

\textit{\textbf{Case 2: No colour set expands.}} Thus we can assume that at some step $i$, our process fails because there exists no $(\theta,k,\ell)$-expanding set $C \subseteq S \setminus C(P_{i-1})$. We will use this fact to find an induced subgraph of our digraph which is dense in a large colour subset of 
\[S' := \left \{\begin{array}{l l}
S & \text{if $i=1$,} \\
S \setminus (C(P_{i-1}) \cup C_{i-1}) & \text{if $i \geq 2$}.
\end{array}\right.\]

Recall from the proof sketch that our goal here is to make a reduction to the dense case, and hence we wish to find a subgraph $G'$ to which we can apply \Cref{theorem:rainbowdirectedpathindensecase}.

\textit{\textbf{Finding $G'$, a dense instance of the problem.}} Recall that $v_{i-1}$ is the terminal vertex of $P_{i-1}$ and that 
\[|S'| \geq |S \setminus C(P_{i-1})| - \theta d \geq \eps d/2 - \theta d \geq \theta d/2\]
since $\theta \ll \eps$. So there exists a set $Z$ consisting of $\lceil \theta d/2 \rceil$ out-neighbours $w$ of $v_{i-1}$ for which $C(v_{i-1},w)\in S'$.  Let $C'$ be the set of colours used on these edges.  If there exists a colour $c \in S \setminus (C(P_{i-1}) \cup C')$ with at least $\gamma'd$ $c$-edges coming out of $Z$, add the terminal vertices of those edges to $Z$ and add $c$ to $C'$. Repeat this process as long as possible. Note that the size of $Z$ cannot ever exceed $kd$. If it did, then this would first occur after at most $kd/(\gamma' d) + 1 \ll \ell$ iterations, so we would have $|C'| \leq \lceil \theta d/2 \rceil + \ell \leq \theta d$ since $1/d \ll 1/\ell, \theta$. Also every vertex in $Z$ would be reachable via a $C'$-rainbow walk of length at most $\ell$ from $v_{i-1}$. Therefore, $C'$ would be $(\theta,k,\ell)$-expanding, a contradiction. Thus, when this process terminates, we have that $\theta d/2 \leq |Z| \leq kd$, that $|C'| \leq \lceil\theta d/2 \rceil + kd/(\gamma' d) \leq \theta d$ since $1/d \ll \gamma', \theta, 1/k$, and that every colour in $C'':= S' \setminus C'$ has at most $\gamma' d$ edges going from $Z$ to $Z^c$. The size of $C''$ is at least $|C''| \geq d -\theta d$ if $i=1$ since $|C'| \leq \theta d$. If $i \geq 2$, then we have $|C''| \geq \eps d/2 - 2\theta d \geq \eps d/4$ since $|C_{i-1}|\leq \theta d$, and $|S\setminus C(P_{i-1})|\geq \eps d/2$, and since $\theta \ll \eps$. From this, we can deduce a stronger lower bound on $|Z|$ than $\theta d/2$; because $\gamma' d \leq \theta d/4 \leq |Z|/2$, we have
\[\eps d|Z|/8 \leq |C''|(|Z| - \gamma' d) \leq e(G[Z;C'']) \leq |Z|^2,\]
and therefore $|Z| \geq \eps d/8$.

Define $r := \sqrt{\gamma'} d$. Recall that for each $c \in C''$, the colour class of $c$ in $G$ is a $1$-factor with at most $\gamma' d$ $c$-edges from $Z$ to $Z^c$, so there are also at most $\gamma' d$ $c$-edges from $Z^c$ to $Z$. Consequently, the set $Y$ of vertices in $Z$ with in-degree or out-degree below $|C''| - r$ in $G[Z;C'']$ has size at most
\[|Y| \leq \frac{2\gamma' d|C''|}{r} = 2\sqrt{\gamma'}|C''| \leq 2r.\]

\textit{\textbf{Verifying the density properties of $G'$. }} Now $G' := G[Z \setminus Y;C'']$ is the dense subgraph we are looking for. In particular, $G'$ satisfies the following properties. Let $n' := |V(G')|$. First,
\[kd \geq n' = |Z| - |Y| \geq \left(\eps/8 - 2\sqrt{\gamma'}\right)d \geq \eps d/16\]
since $\gamma' \ll \eps$. Consequently, 
\[\delta^\pm(G') \geq |C''| - 3r \geq |C''| - 48\left(\sqrt{\gamma'}/\eps\right) n' \geq |C''| - \gamma n' \geq \Delta^\pm(G') - \gamma n'\]
since $\gamma' \ll \gamma,\eps$. 
Also, the size of $C''$ is at least
\[|C''| \geq \eps d/4 \geq \eps n'/(4k) \geq \alpha n' + \gamma n'\]
since $\gamma, \alpha \ll 1/k,\eps$, and so $\delta^\pm(G') \geq \alpha n'$.

\textit{\textbf{Appending a large rainbow path of $G'$ to $P_{i-1}$}}. Since $\gamma \ll \alpha, \eps$, we can apply \Cref{theorem:rainbowdirectedpathindensecase}
with $\eps/4$, $n'$, $G'$ 
in place of $\eps$, $n$, $G$, respectively, to obtain a $C''$-rainbow directed path $Q$ of length at least $(1-\eps/4)\delta^\pm(G') \geq (1-\eps/4)(|C''| - \gamma n')$ in $G'$. If $i=1$, then we can take $P=Q$, which is already a rainbow walk of length at least 
\[(1-\eps/4)((1-\theta) d - \gamma n') \geq (1-\eps/4)(1 - \theta - \gamma k)d \geq (1-\eps/2)d\]
with no vertex repetitions, as $\theta \ll \eps$, and $\gamma \ll 1/k, \eps$.

Otherwise we have $i \geq 2$, and we can append a translate of $Q$ to the walk $P_{i-2}$ found previously. Let $w$ be the initial vertex of $Q$. Recall that by construction, there exists a set $X_{i-1}$ of at least $kd$ vertices such that each vertex $x \in X_{i-1}$ is reachable from $v_{i-2}$ by a $C_{i-1}$-rainbow path of length at most $\ell$.
\begin{claim*}
    There exists $x \in X_{i-1}$ such that $|x w^{-1} V(Q) \cap V(P_{i-2})| \leq \eps d/4$.
\end{claim*}
\begin{cla_proof}
    Similar to the previous claim, we count the number $s$ of triples $(x,u,v) \in X_{i-1} \times V(Q) \times V(P_{i-2})$ such that $xw^{-1}u = v$ in $\Gamma$. Recall that $Q$ and $P_{i-2}$ are both nonempty, rainbow, and colour-disjoint, so each has at most $d-1$ edges, hence at most $d$ vertices. Thus
    \[s \leq |V(Q)||V(P_{i-2})| \leq d^2.\]
    On the other hand, if the claim does not hold, then
    \[s > (\eps d/4)|X_{i-1}| \geq \eps kd^2/4 \geq d^2\]
    since $|X_{i-1}| \geq kd$ and $1/k \ll \eps$, a contradiction.
\end{cla_proof}
Now, we let our walk be $P := P_{i-2}\cup L \cup (xw^{-1}Q)$, where $L$ is a $C_{i-1}$-rainbow path of length at most $\ell$ joining $v_{i-2}$ to $x$. Note that $C(P_{i-2})$, $C_{i-1}$, and $C''$ are disjoint colour sets by design, so $P$ is indeed a rainbow walk. Also note that the length of $P$ is at least
\begin{align*}
    &|C(P_{i-2})|+(1-\eps/4)(|C''| - \gamma n') \\
    &\geq |C(P_{i-1})|-\ell - \lceil \delta d \rceil + (1-\eps/4)(d-|C(P_{i-1})|-2\theta d - \gamma k d) \\
    &\geq (1-\eps/2)d
\end{align*}

since $1/d \ll 1/\ell, \theta, \delta \ll \eps$, and $\gamma \ll 1/k, \eps$.
Finally, $P$ has at most 
\begin{align*}
    |C(P)| - |V(P)| + 1 &\leq |C(P)| - |V(P_{i-2})| - |V(Q)| + |xw^{-1}V(Q) \cap V(P_{i-2})| + 1 \\
    &\leq |C(P_{i-2})| - |V(P_{i-2})| + 1 + |C(L)| + |C(Q)| - |V(Q)| + \eps d/4 \\
    &\leq 2(i-2)\delta d/k + (i-2)\ell + \ell + \eps d/4 \\
    &\leq 2d/k + \delta^{-1}\ell + \eps d/4 \leq \eps d/2
\end{align*}
vertex repetitions since $1/d \ll 1/\ell, 1/k, \delta \ll \eps$.
\end{proof}

\textbf{Remark.} The proof above uses only basic properties of groups and can be easily adapted to any vertex-transitive digraph with a vertex-transitive proper edge colouring; i.e., properly edge-coloured digraphs with the property that for any two vertices $u$ and $v$, there is a colour-preserving automorphism of the digraph mapping $u$ to $v$.

\section{Concluding remarks}\label{sec:concluding}
\par \textbf{Rearrangeable subsets in groups.} In \Cref{thm:summary}(b), we showed that dense coloured Cayley graphs contain directed rainbow paths with a length that is asymptotically best possible. Proving a similar result without a density assumption would be of great interest.

\begin{problem}\label{prob:mainconc}
    For any group $\Gamma$ of order $n$ and any subset $S\subseteq \Gamma$ of size $d$, show that there exists a subset $S' \subseteq S$ of size $d-o(d)$ which is rearrangeable.
\end{problem}
\par Problem~\ref{prob:mainconc} is already open for cyclic groups $\mathbb{Z}_p$ of prime order. In that case, \Cref{thm:summary}(c) finds the desired subset when $d\geq p^{3/4+o(1)}$. By results of \cite{BederdKravitz}, the entire set $S$ is rearrangeable as long as $d\leq e^{(\log p)^{1/4}}$ and $0 \notin S$. We know how to improve the first result to lower density, namely $d \geq p^{2/3 + o(1)}$, with a more involved argument which we chose to omit in this paper for the sake of brevity. However, our methods meet a natural barrier at roughly $d\sim p^{1/2}$, as then the diameter of $\mathrm{Cay}(\mathbb Z_p, S)$ can be larger than the total number of colours available, so a lemma analogous to \Cref{lem:connecting lemma} cannot hold. 

\par \textbf{Rainbow walks in regular digraphs.} Regarding general coloured $d$-regular digraphs, our current methods appear to be too weak to give a positive answer to \Cref{problem:directed}. We pose the following relaxation as a more approachable open problem, with Theorem~\ref{thm:weakasymptotic-intro} already giving a positive answer for coloured Cayley graphs.

\begin{problem}\label{problem:directedwalks} Let $G$ be a $d$-regular digraph properly edge-coloured with $d$ colours. Does $G$ contain a rainbow walk with $d-o(d)$ distinct vertices?
\end{problem}

\par \textbf{Rainbow paths of length asymptotic in minimum degree.} On the other hand, in proving the asymptotic results \Cref{thm:main} and \Cref{theorem:rainbowdirectedpathindensecase}, we were able to drop the condition on the number of colours and relax the regularity condition to some extent. We wonder if a minimum out-degree assumption is enough to guarantee similar results.

\begin{problem}\label{problem:directedmindegree} Let $G$ be  properly edge-coloured digraph with minimum out-degree $d$. Does $G$ contain a rainbow path of length $d-o(d)$?
\end{problem}

Note that \Cref{lem:GreedyPathMinDegree} gives a rainbow path of length at least $d/2$ under these conditions, and we know of no improvement to this simple bound for directed graphs. In the undirected setting, Johnston, Palmer, and Sarkar \cite{johnston2016rainbow} showed that any properly edge-coloured graph of minimum degree $d$ contains a rainbow path of length at least $2d/3$. Chen and Li \cite{chen2005long} considered the more general setting of graphs with arbitrary edge colourings in which every vertex is incident to edges with at least $d$ distinct colours. They proved that in such a regime, it is always possible to find a rainbow path of length at least $2d/3 + 1$ in general (see \cite{babu2015heterochromatic}), and at least $d-1$ if $d \leq 7$. They went on to conjecture that a rainbow path of length $d-1$ always exists. Even proving a bound of $d-o(d)$ in either of these undirected settings would be interesting.

\textbf{Rainbow Tur\'an numbers of paths.} We may also consider analogous problems for \textit{average degree}. Specifically, given $d > 0$, what is the largest integer $f(d)$ such that every properly edge-coloured graph with average degree at least $d$ contains a rainbow path of length $f(d)$? This is equivalent to the so-called \emph{rainbow Tur\'an problem} for paths, formally introduced in 
the influential paper by Keevash, Mubayi, Sudakov and Verstra\"ete \cite{KMSV}. In the uncoloured setting, Erd\H{o}s and Gallai~\cite{gallai1959maximal} showed that any graph with average degree $d$ contains a path of length at least $d$, and this is best possible when $d$ is an integer by considering a disjoint union of cliques of size $d+1$.
In comparison, no tight bound is known for $f(d)$ despite considerable effort in this direction~\cite{halfpap2021rainbow, halfpap2022rainbow, johnston2016rainbow,ergemlidze2018rainbow, rombach-johnston}. Note that \Cref{lem:GreedyOnePath} gives an easy bound of $f(d) \geq d/4$. The current best known lower bound is $f(d) \geq 7d/18 - O(1)$ by Ergemlidze, Gy\H{o}ri and Methuku~\cite{ergemlidze2018rainbow}, and the best upper bound is $f(d) \leq \lceil d \rceil - 1$ for $d > 2$ by Johnston and Rombach~\cite{rombach-johnston}, coming from the coloured Cayley graph $\mathrm{Cay}(\mathbb F_2^k, S)$ generated by a zero-sum set $S$ of size $\lceil d \rceil$. This upper bound is known to give the correct value of $f(d)$ for $2 < d \leq 6$ \cite{johnston2016rainbow, halfpap2022rainbow}, and Halfpap~\cite{halfpap2022rainbow} has conjectured it to be tight for all $d$, up to possibly an additive constant.

The natural analogue of this question for general directed graphs is somewhat less interesting, as directed bicliques can have very high average out-degree but no directed path of length $2$, let alone a rainbow one. To remedy this, one may impose additional assumptions on the digraph. One natural condition to consider is that the digraph is \emph{Eulerian}, meaning that every vertex has in-degree equal to its out-degree.

\begin{problem}What is the maximum length of a rainbow path in any properly edge-coloured Eulerian digraph of average out-degree $d$?
\end{problem}
In the uncoloured setting, a long-standing conjecture of Bollob\'as and Scott (Conjecture 7 in \cite{bollobas-scott}) implies that any Eulerian digraph of average out-degree $d$ contains a path of length $\Omega(d)$. The best known lower bound in this direction is of order $\Omega(d/\log{d})$ by Knierim, Larcher and Martinsson~\cite{KLM}. Our \Cref{lem:find robust expander} might provide some insight into questions such as this, as it gives a way to find robust out-expanders in dense Eulerian digraphs. However, it seems to say nothing without any minimum degree assumptions, so additional ideas are needed.

\textbf{Acknowledgments.} We would like to thank Noga Alon for leading us to Pollard's Theorem, which replaced a more involved earlier argument for establishing expansion properties in $\mathbb{Z}_p$. This paper was completed while Matija Buci\'c, Alp M\"uyesser and Liana Yepremyan were in
residence at the Simons Laufer Mathematical Sciences Institute in
Berkeley, California, during the Spring 2025 semester supported by the National Science
Foundation under Grant No. DMS-1928930. We would like to thank the organizers of the thematic semester and the staff. We would also like to thank the anonymous referees for their useful comments and suggestions, which greatly improved the exposition of the paper.

\providecommand{\MR}[1]{}
\providecommand{\MRhref}[2]{%
\href{http://www.ams.org/mathscinet-getitem?mr=#1}{#2}}

\bibliographystyle{amsplain_initials_nobysame}
\bibliography{bib}

@article{benzing2020long,
  title={Long directed rainbow cycles and rainbow spanning trees},
  author={Benzing, Frederik and Pokrovskiy, Alexey and Sudakov, Benny},
journal={Eur. J. Comb.},  
fjournal={European Journal of Combinatorics},
  volume={88},
  pages={103102},
  year={2020},
  publisher={Elsevier}
}

@article{gruslys2021cycle,
  title={Cycle partitions of regular graphs},
  author={Gruslys, Vytautas and Letzter, Shoham},
  journal={Combinatorics, Probability and Computing},
  volume={30},
  number={4},
  pages={526--549},
  year={2021},
  publisher={Cambridge University Press}
}

@article{kuhn2010hamiltonian,
  title={Hamiltonian degree sequences in digraphs},
  author={K{\"u}hn, Daniela and Osthus, Deryk and Treglown, Andrew},
  journal={Journal of Combinatorial Theory, Series B},
  volume={100},
  number={4},
  pages={367--380},
  year={2010},
  publisher={Elsevier}
}

@article{keevash2022new,
    AUTHOR = {Keevash, Peter and Pokrovskiy, Alexey and Sudakov, Benny and
              Yepremyan, Liana},
     TITLE = {New bounds for {R}yser's conjecture and related problems},
   JOURNAL = {Trans. Amer. Math. Soc. Ser. B},
  FJOURNAL = {Transactions of the American Mathematical Society. Series B},
    VOLUME = {9},
      YEAR = {2022},
     PAGES = {288--321},
      ISSN = {2330-0000},
   MRCLASS = {05D40 (05B15 05C65 05D15)},
  MRNUMBER = {4408405},
MRREVIEWER = {Norihide\ Tokushige},
       DOI = {10.1090/btran/92},
       URL = {https://doi.org/10.1090/btran/92},
}

@article{costa2018problem,
  title={A problem on partial sums in abelian groups},
  author={Costa, Simone and Morini, Fiorenza and Pasotti, Anita and Pellegrini, Marco Antonio},
journal={Discrete Math.},    
fjournal={Discrete Mathematics},
  volume={341},
  number={3},
  pages={705--712},
  year={2018},
  publisher={Elsevier}
}

@article{gordon1961sequences,
    AUTHOR = {Gordon, Basil},
     TITLE = {Sequences in groups with distinct partial products},
   JOURNAL = {Pacific J. Math.},
  FJOURNAL = {Pacific Journal of Mathematics},
    VOLUME = {11},
      YEAR = {1961},
     PAGES = {1309--1313},
      ISSN = {0030-8730,1945-5844},
   MRCLASS = {20.25},
  MRNUMBER = {133359},
MRREVIEWER = {Burton\ W.\ Jones},
       URL = {http://projecteuclid.org/euclid.pjm/1103036916},
}

@article{archdeacon2015partial,
  title={On partial sums in cyclic groups},
  author={Archdeacon, Dan S and Dinitz, Jeffrey H and Mattern, Amelia and Stinson, Douglas R},
  journal={J.
Combin. Math. Combin. Comput.},
  year={2016},
volume={98},
pages={327-342}
}

@article{costa2022sequences,
  title={On sequences in cyclic groups with distinct partial sums},
  author={Costa, Simone and Della Fiore, Stefano and Ollis, MA and Rovner-Frydman, Sarah Z},
  journal={arXiv:2203.16658},
  year={2022}
}

@article{alspach2020strongly,
  title={On strongly sequenceable abelian groups},
  author={Alspach, Brian and Liversidge, Georgina},
  journal={Art Discrete Appl. Math.},
  year={2020}
}

@article{costa2020some,
    AUTHOR = {Costa, S. and Pellegrini, M. A.},
     TITLE = {Some new results about a conjecture by {B}rian {A}lspach},
   JOURNAL = {Arch. Math.},
  FJOURNAL = {Archiv der Mathematik},
    VOLUME = {115},
      YEAR = {2020},
    NUMBER = {5},
     PAGES = {479--488},
      ISSN = {0003-889X,1420-8938},
   MRCLASS = {05C25 (20K15)},
  MRNUMBER = {4154562},
MRREVIEWER = {Song-Tao\ Guo},
       DOI = {10.1007/s00013-020-01507-7},
       URL = {https://doi.org/10.1007/s00013-020-01507-7},
}

@article{schrijver,
  title={Rainbow paths in edge-coloured regular graphs},
  author={Lex Schrijver},
  journal={Manuscript},
}

@article{muyesser2022random,
  title={A random {H}all-{P}aige conjecture},
  author={M{\"u}yesser, Alp and Pokrovskiy, Alexey},
  journal={Invent. Math. (to appear), arXiv:2204.09666},
  year={2022}
}

@article{kuhn2015robust,
  title={The robust component structure of dense regular graphs and applications},
  author={K{\"u}hn, Daniela and Lo, Allan and Osthus, Deryk and Staden, Katherine},
  journal={Proceedings of the London Mathematical Society},
  volume={110},
  number={1},
  pages={19--56},
  year={2015},
  publisher={Oxford University Press}
}

@article{chen2005long,
  title={Long heterochromatic paths in edge-colored graphs},
  author={Chen, He and Li, Xueliang},
journal={Electron. J. Comb.},    
fjournal={The Electronic Journal of Combinatorics},
  pages={R33--R33},
  year={2005}
}

@inproceedings{graham1971sums,
    AUTHOR = {Graham, R. L.},
     TITLE = {On sums of integers taken from a fixed sequence},
 BOOKTITLE = {Proceedings of the {W}ashington {S}tate {U}niversity
              {C}onference on {N}umber {T}heory ({W}ashington {S}tate
              {U}niv., {P}ullman, {W}ash., 1971)},
     PAGES = {22--40},
 PUBLISHER = {Washington State University, Department of Mathematics, Pi Mu
              Epsilon, Pullman, WA},
      YEAR = {1971},
   MRCLASS = {10L10},
  MRNUMBER = {319935},
MRREVIEWER = {P.\ Erd\H os},
}

@article{johnston2016rainbow,
  title={Rainbow {T}ur\'an problems for paths and forests of stars},
  author={Johnston, Daniel and Palmer, Cory and Sarkar, Amites},
journal={Electron. J. Comb.},  
  fjournal={The Electronic Journal of Combinatorics},
volume={24},
issue={1},
  year={2017}
}

@article{babu2015heterochromatic,
  title={Heterochromatic paths in edge colored graphs without small cycles and heterochromatic-triangle-free graphs},
  author={Babu, Jasine and Chandran, L Sunil and Rajendraprasad, Deepak},
journal={Eur. J. Comb.},  
fjournal={European Journal of Combinatorics},
  volume={48},
  pages={110--126},
  year={2015},
  publisher={Elsevier}
}

@inbook{Montgomery_2024, place={Cambridge}, series={London Mathematical Society Lecture Note Series}, title={Transversals in {L}atin Squares}, booktitle={Surveys in Combinatorics 2024}, publisher={Cambridge University Press}, author={Montgomery, Richard}, editor={Fischer, Felix and Johnson, RobertEditors}, year={2024}, pages={131–158}, collection={London Mathematical Society Lecture Note Series}}

@article{rodl1985packing,
  title={On a packing and covering problem},
  author={R{\"o}dl, Vojt{\v{e}}ch},
  journal={European Journal of Combinatorics},
  volume={6},
  number={1},
  pages={69--78},
  year={1985},
  publisher={Elsevier}
}

@article{keevash2014existence,
  title={The existence of designs},
  author={Keevash, Peter},
  journal={arXiv:1401.3665},
  year={2014}
}

@article{hicks2019distinct,
    AUTHOR = {Hicks, Jacob and Ollis, M. A. and Schmitt, John R.},
     TITLE = {Distinct partial sums in cyclic groups: polynomial method and
              constructive approaches},
   JOURNAL = {J. Combin. Des.},
  FJOURNAL = {Journal of Combinatorial Designs},
    VOLUME = {27},
      YEAR = {2019},
    NUMBER = {6},
     PAGES = {369--385},
      ISSN = {1063-8539,1520-6610},
   MRCLASS = {11B75 (05B10)},
  MRNUMBER = {3939799},
MRREVIEWER = {Jamie\ Simpson},
       DOI = {10.1002/jcd.21652},
       URL = {https://doi.org/10.1002/jcd.21652},
}

@article{ryser1967neuere,
  title={Neuere probleme der kombinatorik},
  author={Ryser, Herbert J},
  journal={Vortr{\"a}ge {\"u}ber Kombinatorik, Oberwolfach},
  volume={69},
  pages={91},
  year={1967},
  publisher={Matematisches Forschungsinstitute Oberwolfach, Germany}
}

@book{brualdi1991combinatorial,
  title={Combinatorial matrix theory},
  author={Brualdi, R. A. and Ryser, H. J.},
  volume={39},
  year={1991},
  publisher={Springer}
}

@article{stein1975transversals,
    AUTHOR = {Stein, S. K.},
     TITLE = {Transversals of {L}atin squares and their generalizations},
   JOURNAL = {Pacific J. Math.},
  FJOURNAL = {Pacific Journal of Mathematics},
    VOLUME = {59},
      YEAR = {1975},
    NUMBER = {2},
     PAGES = {567--575},
      ISSN = {0030-8730,1945-5844},
   MRCLASS = {05B15},
  MRNUMBER = {387083},
MRREVIEWER = {Judith\ Q.\ Longyear},
       URL = {http://projecteuclid.org/euclid.pjm/1102905365},
}

@article{Montgomery2024,
author = {Montgomery, R.},
title={A proof of the {R}yser-{B}rualdi-{S}tein conjecture for large even  {$n$}},
year={2024},
note={arXiv:2310.19779}
}

@article{BALOGH2019140,
title = {Long rainbow cycles and {H}amiltonian cycles using many colors in properly edge-colored complete graphs},
journal={Eur. J. Comb.},  
fjournal={European Journal of Combinatorics},
volume = {79},
pages = {140--151},
year = {2019},
issn = {0195-6698},
author = {Balogh, J. and Molla, T.}
}

@article{maamoun1984problem,
  title={On a problem of {G}. {H}ahn about coloured Hamiltonian paths in ${K}_{2t}$},
  author={Maamoun, Malaz and Meyniel, Henry},
journal={Discrete Math.},  
fjournal={Discrete mathematics},
  volume={51},
  number={2},
  pages={213--214},
  year={1984},
  publisher={Elsevier}
}

@article{Hahn, 
author={Hahn, G.},
title={Un jeu de coloration},
journal={Regards sur la theorie des graphes, Actes du Colloque de Cerisy},
volume={12},
page={18}, 
year={1980}
}

@article{andersen1989hamilton,
    AUTHOR = {Andersen, L. D.},
     TITLE = {Hamilton circuits with many colours in properly edge-coloured
              complete graphs},
   JOURNAL = {Math. Scand.},
  FJOURNAL = {Mathematica Scandinavica},
    VOLUME = {64},
      YEAR = {1989},
    NUMBER = {1},
     PAGES = {5--14},
      ISSN = {0025-5521,1903-1807},
   MRCLASS = {05C15 (05C45)},
  MRNUMBER = {1036426},
MRREVIEWER = {David\ E.\ Woolbright},
       DOI = {10.7146/math.scand.a-12245},
       URL = {https://doi.org/10.7146/math.scand.a-12245},
}

@article{alon2017random,
  title={Random subgraphs of properly edge-coloured complete graphs and long rainbow cycles},
  author={Alon, Noga and Pokrovskiy, Alexey and Sudakov, Benny},
  JOURNAL = {Israel J. Math.},
  FJOURNAL = {Israel Journal of Mathematics},
  volume={222},
  pages={317--331},
  year={2017},
  publisher={Springer}
}

@article{approximateringel,
    AUTHOR = {Montgomery, Richard and Pokrovskiy, Alexey and Sudakov,
              Benjamin},
     TITLE = {Embedding rainbow trees with applications to graph labelling
              and decomposition},
   JOURNAL = {J. Eur. Math. Soc. (JEMS)},
  FJOURNAL = {Journal of the European Mathematical Society (JEMS)},
    VOLUME = {22},
      YEAR = {2020},
    NUMBER = {10},
     PAGES = {3101--3132},
      ISSN = {1435-9855,1435-9863},
   MRCLASS = {05C15 (05C51)},
  MRNUMBER = {4153104},
MRREVIEWER = {Anna\ O.\ Ivanova},
       DOI = {10.4171/jems/982},
       URL = {https://doi.org/10.4171/jems/982},
}

@article{ringel,
    AUTHOR = {Montgomery, R. and Pokrovskiy, A. and Sudakov, B.},
     TITLE = {A proof of {R}ingel's conjecture},
   JOURNAL = {Geom. Funct. Anal.},
  FJOURNAL = {Geometric and Functional Analysis},
    VOLUME = {31},
      YEAR = {2021},
    NUMBER = {3},
     PAGES = {663--720},
      ISSN = {1016-443X,1420-8970},
   MRCLASS = {05C70 (05C05 05C51)},
  MRNUMBER = {4311581},
MRREVIEWER = {Martin\ Ba\v ca},
       DOI = {10.1007/s00039-021-00576-2},
       URL = {https://doi.org/10.1007/s00039-021-00576-2},
}

@article{ergemlidze2018rainbow,
  title={On the Rainbow {T}ur\'an number of paths},
  author={Ergemlidze, B. and Gy{\H{o}}ri, E. and Methuku, A.},
journal={Electron. J. Comb.},  
fjournal={Electronic Journal of Combinatorics},
  year={2019},
volume={26},
issue={1},
pages={1--17} 
}

@article{KLM,
    AUTHOR = {Knierim, Charlotte and Larcher, Maxime and Martinsson, Anders},
     TITLE = {Note on long paths in {E}ulerian digraphs},
   JOURNAL = {Electron. J. Combin.},
  FJOURNAL = {Electronic Journal of Combinatorics},
    VOLUME = {28},
      YEAR = {2021},
    NUMBER = {2},
     PAGES = {Paper No. 2.37, 4},
      ISSN = {1077-8926},
   MRCLASS = {05C45 (05C20 05C35 05C38)},
  MRNUMBER = {4281203},
       DOI = {10.37236/10297},
       URL = {https://doi.org/10.37236/10297},
}

@book{evans2018orthogonal,
    AUTHOR = {Evans, Anthony B.},
     TITLE = {Orthogonal {L}atin squares based on groups},
    SERIES = {Developments in Mathematics},
    VOLUME = {57},
 PUBLISHER = {Springer},
      YEAR = {2018},
     PAGES = {xv+537},
      ISBN = {978-3-319-94429-6; 978-3-319-94430-2},
   MRCLASS = {05-02 (05B15 20D60 20N05 51E15)},
  MRNUMBER = {3837138},
MRREVIEWER = {Petr\ Vojt\v echovsk\'y},
       DOI = {10.1007/978-3-319-94430-2},
       URL = {https://doi.org/10.1007/978-3-319-94430-2},
}

@article{friedlander1981partitions,
    AUTHOR = {Friedlander, Richard J. and Gordon, Basil and Tannenbaum,
              Peter},
     TITLE = {Partitions of groups and complete mappings},
   JOURNAL = {Pacific J. Math.},
  FJOURNAL = {Pacific Journal of Mathematics},
    VOLUME = {92},
      YEAR = {1981},
    NUMBER = {2},
     PAGES = {283--293},
      ISSN = {0030-8730,1945-5844},
   MRCLASS = {20K99 (20G99)},
  MRNUMBER = {618066},
MRREVIEWER = {E.\ A.\ Walker},
       URL = {http://projecteuclid.org/euclid.pjm/1102736793},
}

@book{ringel2012map,
    AUTHOR = {Ringel, Gerhard},
     TITLE = {Map color theorem},
    SERIES = {Die Grundlehren der mathematischen Wissenschaften},
    VOLUME = {209},
 PUBLISHER = {Springer-Verlag, New York-Heidelberg},
      YEAR = {1974},
     PAGES = {xii+191},
   MRCLASS = {05C15 (55A15)},
  MRNUMBER = {349461},
MRREVIEWER = {H.\ V.\ Kronk},
}

@article{chung2007universal,
  title={Universal juggling cycles},
  author={Chung, Fan and Graham, Ron},
  journal={Integers},
  volume={7},
  number={2},
  pages={A8},
  year={2007}
}

@article{buhler1994juggling,
    AUTHOR = {Buhler, Joe and Eisenbud, David and Graham, Ron and Wright,
              Colin},
     TITLE = {Juggling drops and descents},
   JOURNAL = {Amer. Math. Monthly},
  FJOURNAL = {American Mathematical Monthly},
    VOLUME = {101},
      YEAR = {1994},
    NUMBER = {6},
     PAGES = {507--519},
      ISSN = {0002-9890,1930-0972},
   MRCLASS = {05A99 (70B99 70D99)},
  MRNUMBER = {1274973},
       DOI = {10.2307/2975316},
       URL = {https://doi.org/10.2307/2975316},
}

@article{graham2013juggling,
    AUTHOR = {Graham, Ron},
     TITLE = {Juggling mathematics and magic},
   JOURNAL = {ICCM Not.},
  FJOURNAL = {ICCM Notices. Notices of the International Consortium of
              Chinese Mathematicians},
    VOLUME = {1},
      YEAR = {2013},
    NUMBER = {1},
     PAGES = {7--9},
      ISSN = {2326-4810,2326-4845},
   MRCLASS = {05A05},
  MRNUMBER = {3155192},
       DOI = {10.4310/ICCM.2013.v1.n1.a3},
       URL = {https://doi.org/10.4310/ICCM.2013.v1.n1.a3},
}

@article{ringeloldproblem,
  title={Cyclic arrangements of the elements of a group.},
  author={Gerhard Ringel},
journal={Notices Amer. Math. Soc.},  
fjournal={Notices of the American Mathematical Society},
  volume={21},
  number={1},
  pages={A-95},
  issue={151},
  year={1974}
}

@article{ollis2002sequenceable,
    AUTHOR = {Ollis, M. A.},
     TITLE = {Sequenceable groups and related topics},
   JOURNAL = {Electron. J. Combin.},
  FJOURNAL = {Electronic Journal of Combinatorics},
    VOLUME = {DS10},
      YEAR = {2002},
     PAGES = {34},
      ISSN = {1077-8926},
   MRCLASS = {20D60 (05B15)},
  MRNUMBER = {4336212},
}

@article{keedwell1981sequenceable,
  title={Sequenceable groups: a survey},
  author={Keedwell, AD and Cameron, PJ and Hirschfeld, JWP and Hughes, DR},
  journal={LMS Lecture Notes},
  volume={49},
  pages={205--215},
  year={1981}
}

@article{BATE2008336,
    AUTHOR = {Bate, S. T. and Jones, B.},
     TITLE = {A review of uniform cross-over designs},
   JOURNAL = {J. Statist. Plann. Inference},
  FJOURNAL = {Journal of Statistical Planning and Inference},
    VOLUME = {138},
      YEAR = {2008},
    NUMBER = {2},
     PAGES = {336--351},
      ISSN = {0378-3758,1873-1171},
   MRCLASS = {99-01},
  MRNUMBER = {2412591},
       DOI = {10.1016/j.jspi.2007.06.008},
       URL = {https://doi.org/10.1016/j.jspi.2007.06.008},
}

@inproceedings{pasotti2022survey,
    AUTHOR = {Pasotti, Anita and Dinitz, Jeffrey H.},
     TITLE = {A survey of {H}effter arrays},
 BOOKTITLE = {New advances in designs, codes and cryptography},
    SERIES = {Fields Inst. Commun.},
    VOLUME = {86},
     PAGES = {353--392},
 PUBLISHER = {Springer},
      YEAR = {2024},
      ISBN = {978-3-031-48678-4; 978-3-031-48679-1},
   MRCLASS = {05B15 (05B20)},
  MRNUMBER = {4769493},
       DOI = {10.1007/978-3-031-48679-1\_20},
       URL = {https://doi.org/10.1007/978-3-031-48679-1_20},
}

@article{alon1999combinatorial,
  title={Combinatorial {N}ullstellensatz},
  author={Alon, Noga},
   JOURNAL = {Combin. Probab. Comput.},
  FJOURNAL = {Combinatorics, Probability and Computing},
  volume={8},
  number={1-2},
  pages={7--29},
  year={1999},
  publisher={Cambridge University Press}
}

@article{BederdKravitz, 
Author={Bedert, Benjamin and Kravitz, Noah},
title={Graham's rearrangement conjecture beyond the rectification barrier},
note={Isr. J. Math. (to appear), arXiv:2409.07403},
year={2024}}

@article{AlpSolo,
title={{Cycle type in Hall-Paige: A proof of the
Friedlander-Gordon-Tannenbaum conjecture}},
author={M\"uyesser, Alp},
note={arXiv:2303.16157},
year={2023}
}

@article{gallai1959maximal,
  title={On maximal paths and circuits of graphs},
  author={Erd{\H{o}}s, P and Gallai, T},
  journal={Acta Math. Acad. Sci. Hungar},
  volume={10},
  pages={337--356},
  year={1959}
}

@article{halfpap2022rainbow,
  title={{The rainbow Tur\'an number of $ P_5$}},
  author={Halfpap, Anastasia},
   journal={Australas. J. Combin.},
  year={2023},
volume={87},
issue={3},
pages={403–422}
}

@article{halfpap2021rainbow,
  title={Rainbow cycles vs. rainbow paths},
  author={Halfpap, Anastasia and Palmer, Cory},
  journal={Australas. J. Combin.},
  volume={81},
  pages={152-169},
  year={2021}
}

@article{gyarfas2014rainbow,
    AUTHOR = {Gy\'arf\'as, Andr\'as and S\'ark\"ozy, G\'abor N.},
     TITLE = {Rainbow matchings and cycle-free partial transversals of
              {L}atin squares},
   JOURNAL = {Discrete Math.},
  FJOURNAL = {Discrete Mathematics},
    VOLUME = {327},
      YEAR = {2014},
     PAGES = {96--102},
      ISSN = {0012-365X,1872-681X},
   MRCLASS = {05C15 (05B15 05C70)},
  MRNUMBER = {3192419},
MRREVIEWER = {Chris\ Rodger},
       DOI = {10.1016/j.disc.2014.03.010},
       URL = {https://doi.org/10.1016/j.disc.2014.03.010},
}

@article{Kravitz, 
Author={Kravitz, Noah},
title={Rearranging small sets for distinct partial sums},
note={arXiv:2409.07403},
year={2024}}

@book{ErdosGraham,
  title={Old and new problems and results in combinatorial number theory},
  author={Erd{\H{o}}s, Paul and Graham, Ronald L},
  year={1980},
  publisher={Universit{\'e} de Gen{\`e}ve, L’Enseignement Math{\'e}matique, Geneva}
}

@article{keevash-staden-ringel,
  title={Ringel's tree packing conjecture in quasirandom graphs},
  author={Keevash, Peter and Staden, Katherine},
journal={J. Eur. Math. Soc. (JEMS)}, 
fjournal={Journal of the European Mathematical Society},
  pages={In--Press, arXiv:2004.09947},
  year={2023}
}

@article{sudakov2024restricted,
  title={Restricted subgraphs of edge-colored graphs and applications},
  author={Sudakov, Benny},
  journal={arXiv:2412.13945},
  year={2024}
}

@article {KMSV,
    AUTHOR = {Keevash, Peter and Mubayi, Dhruv and Sudakov, Benny and
              Verstra\"{e}te, Jacques},
     TITLE = {Rainbow {T}ur\'{a}n problems},
   JOURNAL = {Combin. Probab. Comput.},
  FJOURNAL = {Combinatorics, Probability and Computing},
    VOLUME = {16},
      YEAR = {2007},
    NUMBER = {1},
     PAGES = {109--126},
      ISSN = {0963-5483},
   MRCLASS = {05C35},
  MRNUMBER = {2286514},
MRREVIEWER = {Ivan Pashov},
       DOI = {10.1017/S0963548306007760},
       URL = {https://doi.org/10.1017/S0963548306007760},
}

@article {rombach-johnston,
    AUTHOR = {Johnston, Daniel and Rombach, Puck},
     TITLE = {Lower bounds for rainbow {T}ur\'an numbers of paths and other
              trees},
   JOURNAL = {Australas. J. Combin.},
  FJOURNAL = {The Australasian Journal of Combinatorics},
    VOLUME = {78},
      YEAR = {2020},
     PAGES = {61--72},
      ISSN = {1034-4942,2202-3518},
   MRCLASS = {05C35 (05C15)},
  MRNUMBER = {4148505},
MRREVIEWER = {Ioan\ Tomescu},
}

@article {rainbow-tomon,
    AUTHOR = {Tomon, Istv\'an},
     TITLE = {Robust (rainbow) subdivisions and simplicial cycles},
   JOURNAL = {Adv. Comb.},
  FJOURNAL = {Advances in Combinatorics},
      YEAR = {2024},
     PAGES = {Paper No. 1, 37},
      ISSN = {2517-5599},
   MRCLASS = {05C35 (05E45)},
  MRNUMBER = {4695962},
}

@article {erdos-gallai,
    AUTHOR = {Buci\'c, Matija and Montgomery, Richard},
     TITLE = {Towards the {E}rd{\H{o}}s-{G}allai cycle decomposition
              conjecture},
   JOURNAL = {Adv. Math.},
  FJOURNAL = {Advances in Mathematics},
    VOLUME = {437},
      YEAR = {2024},
     PAGES = {Paper No. 109434, 40},
      ISSN = {0001-8708,1090-2082},
   MRCLASS = {05C51},
  MRNUMBER = {4672586},
MRREVIEWER = {Donglei\ Yang},
       DOI = {10.1016/j.aim.2023.109434},
       URL = {https://doi.org/10.1016/j.aim.2023.109434},
}

@article{ALSPACH200177,
title = {Cycle Decompositions of {$K_n$} and {$K_{n-1}$}},
journal = "J. Comb. Theory Ser. B",
fjournal = {Journal of Combinatorial Theory, Series B},
volume = {81},
number = {1},
pages = {77-99},
year = {2001},
issn = {0095-8956},
author = {Alspach, B. and Gavlas,H.}
}

@article{ALSPACH2003165,
title = {Cycle decompositions {IV}: complete directed graphs and fixed length directed cycles},
journal = "J. Comb. Theory Ser. A",
fjournal = {Journal of Combinatorial Theory, Series A},
volume = {103},
number = {1},
pages = {165-208},
year = {2003},
issn = {0097-3165},
author = {Alspach,B. and  Gavlas, H. and {\v{S}}ajna, M. and  Verrall. H.}
}

@article{vsajna2002cycle,
  title={Cycle decompositions {III}: complete graphs and fixed length cycles},
  author={{\v{S}}ajna, Mateja},
journal = {J. Combin. Des.},  
fjournal={Journal of Combinatorial Designs},
  volume={10},
  number={1},
  pages={27--78},
  year={2002},
  publisher={Wiley Online Library}
}

@article{ConlonHaenni,
author={Conlon, D. and Haenni, K.}, 
note={personal communication}
}

@article{vsajna2003decomposition,
  title={Decomposition of the complete graph plus a 1-factor into cycles of equal length},
  author={{\v{S}}ajna, Mateja},
journal = {J. Combin. Des.},  
fjournal={Journal of Combinatorial Designs},
  volume={11},
  number={3},
  pages={170--207},
  year={2003},
  publisher={Wiley Online Library}
}

@article{BodeHarborth, 
author={Bode, J.P. and Harborth, H.},
title={Directed paths of diagonals within polygons},
journal={Discrete Math.},
volume={299},
year={2005}, 
pages={3--10}}

@article{Sawin2015,
author={W. Sawin}, 
note={comment on the post ``Ordering subsets of the cyclic group to give distinct partial sums''},
journal={MathOverflow}, 
url={https://mathoverflow.net/q/202857},
year={2015}
}

@article {pollard,
    AUTHOR = {Pollard, J. M.},
     TITLE = {A generalisation of the theorem of {C}auchy and {D}avenport},
   JOURNAL = {J. London Math. Soc. (2)},
  FJOURNAL = {Journal of the London Mathematical Society. Second Series},
    VOLUME = {8},
      YEAR = {1974},
     PAGES = {460--462},
      ISSN = {0024-6107,1469-7750},
   MRCLASS = {10A10 (10L05)},
  MRNUMBER = {354517},
MRREVIEWER = {H.\ Halberstam},
       DOI = {10.1112/jlms/s2-8.3.460},
       URL = {https://doi.org/10.1112/jlms/s2-8.3.460},
}

@article {pollard-generalized,
    AUTHOR = {Green, Ben and Ruzsa, Imre Z.},
     TITLE = {Sum-free sets in abelian groups},
   JOURNAL = {Israel J. Math.},
  FJOURNAL = {Israel Journal of Mathematics},
    VOLUME = {147},
      YEAR = {2005},
     PAGES = {157--188},
      ISSN = {0021-2172,1565-8511},
   MRCLASS = {11B75 (11B05)},
  MRNUMBER = {2166359},
MRREVIEWER = {Mei\ Chu\ Chang},
       DOI = {10.1007/BF02785363},
       URL = {https://doi.org/10.1007/BF02785363},
}

@article{Grynkiewicz2013,
author="Grynkiewicz, David J.",
title="Pollard's {T}heorem for {G}eneral {A}belian {G}roups",
  journal={Structural Additive Theory},
  pages={155--179},
  year={2013},
  publisher={Springer}
}

@incollection {alexey-survey,
    AUTHOR = {Pokrovskiy, Alexey},
     TITLE = {Rainbow subgraphs and their applications},
 BOOKTITLE = {Surveys in combinatorics 2022},
    SERIES = {London Math. Soc. Lecture Note Ser.},
    VOLUME = {481},
     PAGES = {191--214},
 PUBLISHER = {Cambridge Univ. Press, Cambridge},
      YEAR = {2022},
      ISBN = {978-1-009-09622-5},
   MRCLASS = {05C55 (05C15)},
  MRNUMBER = {4421403},
}

@article {bollobas-scott,
    AUTHOR = {Bollob\'as, B. and Scott, A. D.},
     TITLE = {A proof of a conjecture of {B}ondy concerning paths in
              weighted digraphs},
   JOURNAL = {J. Combin. Theory Ser. B},
  FJOURNAL = {Journal of Combinatorial Theory. Series B},
    VOLUME = {66},
      YEAR = {1996},
    NUMBER = {2},
     PAGES = {283--292},
      ISSN = {0095-8956,1096-0902},
   MRCLASS = {05C38},
  MRNUMBER = {1376051},
MRREVIEWER = {R.\ C.\ Entringer},
       DOI = {10.1006/jctb.1996.0021},
       URL = {https://doi.org/10.1006/jctb.1996.0021},
}

\end{document}